\documentclass{amsart}
\usepackage{amsmath}
\usepackage{amssymb}
\usepackage{amsthm}
\usepackage{dsfont}
\usepackage{lineno}
\usepackage{subfigure}
\usepackage{graphicx}
\usepackage{multirow}
\usepackage{color}
\usepackage{xspace}
\usepackage{pdfsync}
\usepackage{hyperref} 
\usepackage{algorithmicx}
\usepackage{algpseudocode}
\usepackage{algorithm}
\usepackage{mathtools}
\usepackage{enumitem}
\usepackage{comment}
\graphicspath{{Figures/}}

\newcommand{\bq}{\begin{equation}}
\newcommand{\eq}{\end{equation}}

\algnewcommand{\LineComment}[1]{\State \(\triangleright\) #1}

\newtheorem{theorem}{Theorem}

\theoremstyle{lemma}
\newtheorem{lemma}[theorem]{Lemma}

\newtheorem{corollary}[theorem]{Corollary}

\newtheorem{definition}[theorem]{Definition}

\newtheorem{remark}[theorem]{Remark}

\newtheorem{hypothesis}[theorem]{Hypothesis}

\theoremstyle{remark}

\makeatletter
\newcommand\appendix@section[1]{%
\refstepcounter{section}%
\orig@section*{Appendix \@Alph\c@section: #1}%
}
\let\orig@section\section
\g@addto@macro\appendix{\let\section\appendix@section}
\makeatother

\begin{document}

\title[Point to Far-Field Lens]{Optimal Transport with Defective Cost Functions with Applications to the Lens Refractor Problem}

\author{Axel G. R. Turnquist}
\address{Department of Mathematics, University of Texas at Austin, Austin, TX, 78712}
\email{agrt@utexas.edu}

\begin{abstract}
We define and discuss the properties of a class of cost functions on the sphere which we term defective cost functions. We then discuss how to extend these definitions and some properties to cost functions defined on Euclidean space and on surfaces embedded in Euclidean space. Some important properties of defective cost functions are that they result in Optimal Transport mappings which map to points along geodesics, have a nonzero mixed Hessian term, among other important properties. We also compute the cost-sectional curvature for a broad class of cost functions, to verify and some known examples of cost functions and easily prove positive cost-sectional curvature for some new cost functions. Finally, we discuss how we can construct a regularity theory for defective cost functions by satisfying the Ma-Trudinger-Wang (MTW) conditions on an appropriately defined domain. As we develop the regularity theory of defective cost functions, we discuss how the results apply to a particular instance of the far-field lens refractor problem.
\end{abstract}

\date{\today}    
\maketitle

\section{Introduction}\label{sec:introduction}
Recently, much work has been done on deriving PDE formulations of freeform optics problems with lenses and reflectors, see for example the work in~\cite{YadavThesis, Wang_Reflector, Wang_Reflector2, unified, gutierrezhuang}, resulting in Optimal Transport PDE and, more generally, generated Jacobian equations posed on the unit sphere $\mathbb{S}^2 \subset \mathbb{R}^3$. The general problem is to find, given source and target light intensities modeled as probability measures $\mu \subset \mathbb{S}^2$ and $\nu \subset \mathbb{S}^2$, the function that determines the shape of a lens (or reflector) system that achieves the desired light intensity output pattern. The PDE formulations for some of these optics problems have an important interpretation in terms of Optimal Transport. On the theoretical side, some interest in the menagerie of Optimal Transport problems that arise from these optics problems stems from the rather ``exotic" nature of these cost functions in their Optimal Transport interpretation. For optical problems concerning redirecting directional light to a desired far-field intensity, the it is natural for the PDE to be posed on the sphere. Some cost functions on the sphere which appear ``exotic" actually satisfy some very natural conditions and thus lead to well-behaved Optimal Transport mappings. A regularity theory has been built for ``well-behaved" cost functions in Theorem 4.1 of Loeper~\cite{Loeper_OTonSphere}. The cost functions we address in this manuscript do not fit into the theory established by Loeper, but instead exhibit ``defects" that do not allow for mass to be transported too far on the unit sphere, hence we deem these cost functions ``defective". However, not all is lost, and defective cost functions can satisfy the Ma-Trudinger-Wang (MTW) conditions (see~\cite{MTW}) on appropriate domains.

We will pay special attention to a problem known as the far-field lens refractor problem. The setup is as follows: directional light gets redirected through a refractive medium to a desired directional target intensity. The refractive medium could have 1) refractive index less than or 2) greater than the surrounding medium, resulting in different cost functions in their Optimal Transport PDE formulation. These cases are very different on the PDE level, mainly due to the fact that in one case the cost-sectional curvature is positive and in the other, negative. However, both cost functions exhibit the same problem in that they do not allow for mass to be transported too far. Most of the early work on the refractor problem, especially proving the existence of weak solutions of the problem and computing the cost-sectional curvature was done in~\cite{gutierrezhuang} and further related optical problems (that do not have Optimal Transport PDE formulations) were developed by Gutierrez and coauthors in a series of works, for example the near-field refractor problem work in~\cite{ghnear}.

The goal of this manuscript, is twofold. First, and foremost, is a greater exploration of the specific PDE arising in the far-field lens refractor problem and its solvability and regularity, with an eye for designing future numerical discretizations which ideally should take into account the solvability condition and yield provably convergent solutions even when the solution of the PDE are known to be ``weak" viscosity solutions. Secondly, our goal is to develop a theory defective cost functions to facilitate the analysis of such Optimal Transport cost functions arising in optics problem. We can characterize many cost functions by in which direction the mass transports and how far the mass is ``allowed" to travel given some structure on the cost function. Defective cost functions will transport mass along geodesics and will not allow mass to be transported beyond a certain distance. Thus, this manuscript is also an exploration into the structure of such cost functions and thus many of the results in this paper do not exclusively apply to the cost function arising from the lens refractor problem, but apply in greater generality. This allows for the construction of a regularity theory that does not rely on arguments involving ellipsoids or hyperboloids of revolution, which was the approach taken by~\cite{gutierrezhuang} and~\cite{Karakhanyan}, but rather follows in the footsteps of the line of work by Loeper in~\cite{LoeperReg}, and especially the paper~\cite{Loeper_OTonSphere}, which allows for the theory to apply to a wide variety of cost functions. As a byproduct of our work, for a large class of cost functions, we have discovered more fundamental conditions which will allow for Theorem 4.1 of Loeper's work~\cite{Loeper_OTonSphere} to apply, as opposed to posing the MTW conditions as an assumption in that theorem. This allows, then, for a much easier way for a researcher working on applications to verify the hypotheses of Theorem 4.1 of Loeper for many cost functions, since the computations are done in full in this manuscript. An example of using the theorems in this manuscript to verify the MTW conditions is shown in Section~\ref{sec:MTW}.

In this paper, our goal is to analyze solvability and regularity of such PDE for a class of cost functions which we will call defective cost functions. We desire to prove that smooth enough data lead to smooth solutions. In Section~\ref{sec:background}, we review the setup of the lens refractor problem, which serves as an example to motivate the types of cost functions we will be considering.  Section~\ref{sec:computations} contains the core computations and definitions of this paper. We motivate the definition of defective cost functions via a computation that shows a solvability condition in the lens refractor problem. We then show an example where the lens refractor problem cannot be solved. We then introduce the conditions required to have cost functions leading to exponential-type maps for Optimal Transport problems on $\mathbb{S}^2 \subset \mathbb{R}^3$, $\mathbb{R}^d$ and, surfaces $M \subset \mathbb{R}^3$. We introduce a function $\beta$, which characterizes how far the mapping moves along a geodesic. We then introduce the definition of a defective cost function, which is a cost function that changes concavity as a function of the Riemannian distance. We show that this is the condition for the lens refractor problem that leads to the solvability condition observed earlier. Using these definitions and some useful formulas, we show how defective cost functions have nonzero mixed Hessian determinant on the sphere $\mathbb{S}^2$ and $\mathbb{R}^d$. We end this section with a computation of the cost-sectional curvature for defective cost functions on $\mathbb{S}^2$ and $\mathbb{R}^d$ and apply the results to some known examples, confirming previous works in the literature. In Section~\ref{sec:regularity}, we first show how defective cost functions satisfy the MTW conditions, with the caveat that they must also, independently, satisfy the cost-sectional curvature condition. We then show a computation of power cost functions on the unit sphere and how they satisfy every condition in Theorem 4.1 of Loeper~\cite{Loeper_OTonSphere}, except for the strict cost-sectional curvature condition, unless we are dealing with the squared geodesic cost. We then show how one can control the distance the Optimal Transport mapping moves mass by imposing some controls on the source and target density. This allows us to build a $C^{\infty}$ regularity theory for defective cost functions following the outline of the regularity theory in Loeper~\cite{Loeper_OTonSphere}. In Section~\ref{sec:conclusion} we summarize the contributions of the paper.

\section{Preliminaries}\label{sec:background}
In this section, we first introduce the far-field lens refractor problem, in order to introduce a known physical application where defective cost functions arise. The definition of defective cost functions appears in Definition~\ref{def:defective}. We then show a simple example where the lens refractor problem cannot be solved and argue in simple terms why this should make sense given the physical setup of the problem.

\subsection{Far-Field Lens Refractor Problem}\label{sec:P2FFL}




\subsubsection{Interior and Exterior Media Separated By Outer Lens Edge}\label{sec:lens2}

The far-field refractor problem is described in~\cite{gutierrezhuang} as follows. Light radiating from a point in an inner medium with refractive index $n_{I}$ goes in the direction $x$ and then passes through a lens with refractive index $n_{O}$, whose interior edge is a sphere and whose outer edge is described by the function $\mathcal{L} = u_1(x)x$ and enters the outer medium also with refractive index $n_{O}$ in the direction $y$. Light radiating from the origin has a radial intensity pattern $f(x)$. After passing through the lens, the far-field intensity pattern is given by $g(y)$, see Figure~\ref{fig:lens2}. The inverse problem, as in Section~\ref{sec:lens1} is to solve for the function $u_1(x)$ given $f$ and $g$.

\begin{figure}[htp]
	\centering
	\includegraphics[width=0.7\textwidth]{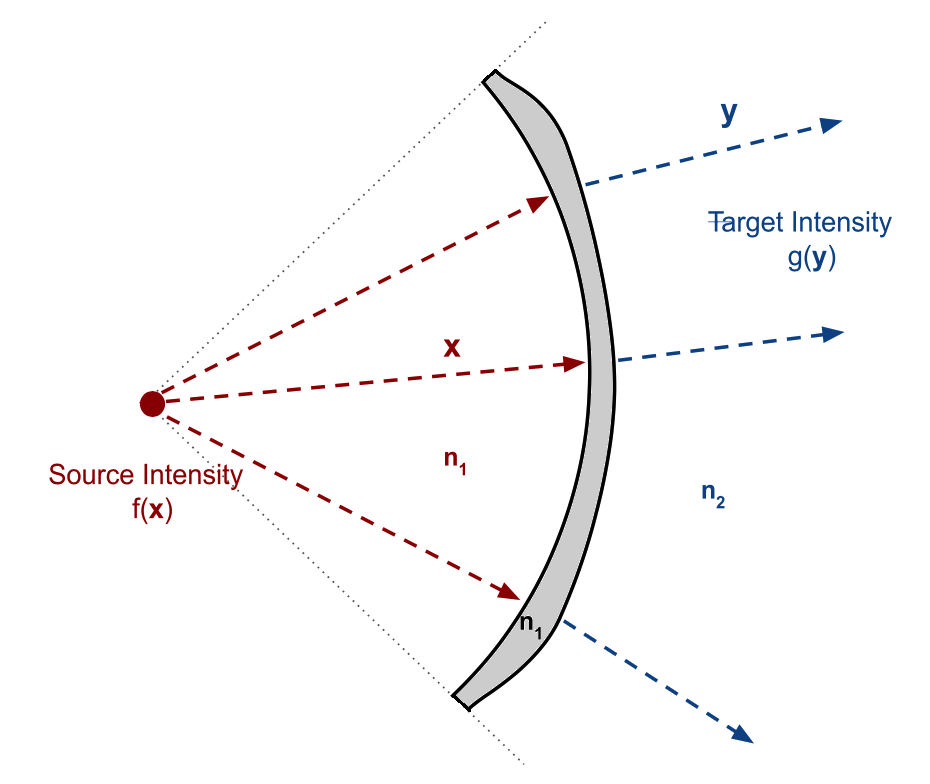}
	\caption{Light with radial intensity $f(x)$ travels in the direction $x$ in a region of refractive index $n_{I}$, enters the lens with refractive index $n_{O}$, whose inside edge is spherical and outer edge is given by the function $\mathcal{L} = u_1(x)x$, then exits the lens into a region of refractive index $n_{O}$ in the direction $y$ making a far-field intensity pattern $g(y)$.}
	\label{fig:lens2}
\end{figure}

For reference, the index of refraction of a vacuum is $1$, water is $1.333$, glass is $1.52$, and diamond is $2.417$. These representative values will be used throughout the paper. The setup of the lens refractor problem allows for us to define two quantities, $n=n_{I}/n_{O}$, and $\kappa = n_{O}/n_{I}$. We choose to use different notation, since the PDE formulations of each problem will be different. We will refer to the case where $n=n_{I}/n_{O}>1$ as lens refractor problem I and the case with $\kappa>1$ as lens refractor problem II.

\subsubsection{Potential and Cost Functions}

From~\cite{point2fflens}, the derivation of the PDE for the lens refractor problem I shows that the cost function is of the following form:

\begin{equation}\label{eq:lensp2ff}
c(x,y) = -\log \left( n - x \cdot y \right).
\end{equation}

Such a cost function may seem a bit unusual in Optimal Transport contexts, but it very closely resembles the cost function for the perhaps better known reflector antenna problem $c(x,y) = -\log(1 - x \cdot y)$, see~\cite{Wang_Reflector, Wang_Reflector2}. Some clear differences are, however, that the cost in~\eqref{eq:lensp2ff} is Lipschitz, whereas the cost $c(x,y) = -\log(1 - x \cdot y)$ for the reflector antenna problem is not. Also, when the magnitude of the gradient of the potential function is zero, the mapping for~\eqref{eq:lensp2ff} satisfies $T(x) = x$, whereas for the reflector antenna we get $T(x) = -x$. It should become clear as we analyze the lens problem in greater detail, then, that it does not make sense to take the limit as $n \rightarrow 1$ since we cannot talk about the lens problem ``converging" in any sense to the reflector problem. In fact, in some sense the reflector antenna cost is less problematic, even though it is not Lipschitz, and the associated regularity theory for the solution of the resulting Optimal Transport PDE was addressed in~\cite{Loeper_OTonSphere}.

From~\cite{gutierrezhuang}, the derivation of the PDE for the lens refractor problem II shows that the cost function is of the following form:

\begin{equation}\label{eq:kappacost}
c(x,y) = \log \left( \kappa x \cdot y - 1 \right).
\end{equation}

We can immediately notice a potential cause of concern. Since $\kappa>1$, it is possible on the sphere to have $x \cdot y =1/\kappa$. That means that the cost function in Equation~\eqref{eq:kappacost} is not Lipschitz. Moreover, there is an issue if in our Optimal Transport formulation, we have source and target densities which require that mass transports further than a distance $\arccos(1/\kappa)$. Surprisingly, as will be seen in Section~\ref{sec:costsectionalcurvature}, reflector antenna problem II is actually the less problematic of the two, even though at first glance its issues seem more apparent. 

\subsubsection{PDE Formulation}

From the conservation of light intensity, for any subset $A \subset \mathbb{S}^2$ where $T(A) \subset \mathbb{S}^2$ we get

\begin{equation}\label{eq:conservation}
\int_{A} f(x) dx = \int_{T(A)} g(y) dy.
\end{equation}

Solving for the mapping using the law of refraction (given the shape of the outer edge of the lens $u_1(x)$) and using the change of variables $u(x) = \log u_1(x)$ for the lens refractor problem I and the change of variables $u(x) = -\log u_1(x)$ for the lens refractor problem II, the following PDE is derived

\begin{equation}\label{eq:thePDE}
\det \left( D^2 u(x) + D^2_{xx} c(x,y) \vert_{y = T(x)} \right) = \frac{\left\vert D^2_{xy} c(x,y) \vert_{y = T(x)} \right\vert f(x)}{g(T(x))},
\end{equation}
where
\begin{equation}\label{eq:mapping}
\nabla u(x) = -\nabla_{x} c(x,y) \vert_{y = T(x)}.
\end{equation}

\section{Computations, Defective Cost Functions, Mixed Hessian Determinants, and Cost-Sectional Curvature}\label{sec:computations}
\subsection{Explicit Form of the Mapping for the Lens Reflector Problem}

For $x, y \in \mathbb{S}^2$, we initiate our study by performing an explicit computation of the mapping in terms of the gradient of the potential function $\nabla u$, which we label as $p \in T_{x}\mathbb{S}^2$. This will inform us about when the PDE~\eqref{eq:thePDE}, and hence the far-field lens refractor problem, will be unsolvable. We solve for the mapping $T = y$ via the equation $\nabla_{\mathbb{S}^2,x} c(x,y) = -p$. Let $\hat{q} = x \times \hat{p}$.

\subsubsection{Computation of Mapping and Solvability Condition for the Lens Refractor Problem I}

For $c(x,y) = -\log \left(n - x \cdot y \right)$, we get, in local tangent coordinates $(\hat{p}, \hat{q})$:

\begin{equation}
\left( \left\Vert p \right\Vert, 0 \right) = \left( \frac{-y \cdot \hat{p}}{n - y \cdot x} , \frac{-y \cdot \hat{q} }{n - y \cdot x} \right).
\end{equation}

Since $n - x \cdot y \neq 0$, we find that $y \cdot \hat{q} =0$. Thus, along with the constraint $(y \cdot x)^2 + (y \cdot \hat{p})^2 = 1$, we arrive at the following expression for the mapping $T$:

\begin{equation}\label{eq:lensmapping}
T(x,p) = x \ \frac{\sqrt{1 + (1-n^2) \left\Vert p \right\Vert^2}+n \left\Vert p \right\Vert^2}{1 + \left\Vert p \right\Vert^2} + p \ \frac{\sqrt{1 + (1-n^2) \left\Vert p \right\Vert^2}-n}{1 + \left\Vert p \right\Vert^2}.
\end{equation}

Note that $p=0 \implies T(x,0) = x$. Since the shape of the reflector satisfies $\nabla u(x) = \nabla u_1(x) e^{u_{1}(x)}$ then $p = \nabla u_1(x) = 0$ implies the lens is ``flat" with respect to the canonical spherical metric, which, of course, makes sense physically because where the lens is flat the light is not redirected but simply passes straight through.

An important observation from Equation~\eqref{eq:lensmapping} is that $\left\Vert p \right\Vert$ cannot exceed a certain threshold without making the mapping become complex valued. This consequently puts a hard constraint on how far the mapping $T(x)$ can move mass at a point $x$. We designate the quantity $p^{*} = 1/\sqrt{n^2-1}$ and any associated vector with magnitude $p^{*}$ as $p^{*}\hat{p}$ and at this value the square roots vanish in Equation~\eqref{eq:lensmapping}. Therefore,

\begin{equation}
T(x,p^{*}\hat{p}) = \frac{n}{1+(p^{*})^2} \left( x (p^{*})^2 - p^{*}\hat{p} \right),
\end{equation}

\begin{equation}
= \frac{x}{n} - \frac{\sqrt{n^2-1}}{n} \hat{p}.
\end{equation}

Thus,

\begin{equation}
T(x, p^{*} \hat{p}) \cdot x = \frac{1}{n},
\end{equation}
and thus denoting $y = T$, we see that we get the following solvability condition for the case when $n=n_{O}/n_{I} >1$:

\begin{equation}\label{eq:lensrefractorconstraint}
x \cdot y \geq \frac{1}{n}.
\end{equation}












The requirement in Equation~\eqref{eq:lensrefractorconstraint} has been noted in previous works, such as~\cite{gutierrezhuang}. Clearly, if this condition is not satisfied, we cannot find an Optimal Transport mapping that solves the PDE~\eqref{eq:thePDE} . From the results in~\cite{gutierrezhuang}, we actually see that this solvability condition is both necessary and sufficient for the existence of solutions of the PDE~\eqref{eq:thePDE} (in a weak sense). In Section~\ref{sec:defective}, we will show how we work the solvability condition into the definition of a defective cost function. 

\subsubsection{Computation of Mapping and Solvability Condition for the Lens Refractor Problem II}

For the cost function $c(x,y) = \log \left( \kappa x \cdot y - 1 \right)$, we immediately note that for the cost function to be evaluated we require that
\begin{equation}\label{eq:restriction2}
x \cdot y \geq 1/\kappa.
\end{equation}

This condition also appears when we solve for the mapping $T$. Following the same procedure as above, we derive the following expression for the mapping $T$:

\begin{equation}\label{eq:kappamapping}
T(x,p) = x\frac{\sqrt{\kappa^2+(\kappa^2- 1) \left\Vert p \right\Vert^2} + \left\Vert p \right\Vert^2}{\kappa(1 + \left\Vert p \right\Vert^2)} + p \frac{\sqrt{\kappa^2+(\kappa^2- 1) \left\Vert p \right\Vert^2} - 1}{\kappa(1 + \left\Vert p \right\Vert^2)}.
\end{equation}

By taking the limit $\left\Vert p \right\Vert \rightarrow 0$, we get $T(x) = x$. By taking the limit, $\left\Vert p \right\Vert \rightarrow \infty$, we get

\begin{equation}
\lim_{\left\Vert p \right\Vert \rightarrow \infty} T(x,p) = \frac{1}{\kappa}x + \hat{p} \frac{\sqrt{\kappa^2-1}}{\kappa},
\end{equation}
and thus $x \cdot y = \frac{1}{\kappa}$. Thus, we do not have any restrictions on the magnitude of $\left\Vert p \right\Vert$, but mass cannot be transported beyond a certain distance no matter what the potential function looks like.

\begin{remark}
What we will show later in this section is that the solvability condition from Equation~\eqref{eq:lensrefractorconstraint} and the solvability condition from Equation~\eqref{eq:restriction2}, while they look very similar, arise from two different properties of the cost function $c(x,y)$, which we will encapsulate in the definition of defective cost function in Section~\ref{sec:defective}. Both cost functions can be written as $c(x,y) = G(d_{\mathbb{S}^2}(x,y))$. The condition from Equation~\eqref{eq:lensrefractorconstraint} arises because $G''(z) = 0$ for $z = \arccos(1/n)$. The condition from Equation~\eqref{eq:restriction2} arises because $G'(z) = \infty$ for $z = \arccos(1/\kappa)$.
\end{remark}

\subsection{A Simple Example Showing Ill-Posedness}\label{sec:illposedness}

It is now apparent that for the far-field refractor problem we must have that the Optimal Transport mapping cannot move mass beyond a certain distance. This puts a hard constraint on the allowable source and target masses. Here, we give an example demonstrating a source and target intensities that do not satisfy the solvability condition. Depicted schematically in Figure~\ref{fig:dirac}, no single lens of can take the mass in red to the mass in blue, no matter how large the refractive index is. Physically, this makes sense that it would be unsolvable since the maximum angle at which the light rays can be deflected by a lens is a right angle, by taking the limit in the refractor problem I, i.e. $\lim_{n \rightarrow \infty} \arccos(1/n) = \pi/2$ from Equation~\eqref{eq:lensrefractorconstraint} or the limit in the refractor problem II, i.e. $\lim_{\kappa \rightarrow \infty} \arccos(1/\kappa) = \pi/2$ from Equation~\eqref{eq:restriction2}. However, the light intensity in Figure~\ref{fig:dirac} must transport more than halfway around the sphere.

\begin{figure}[htp]
	\centering
	\includegraphics[width=\textwidth]{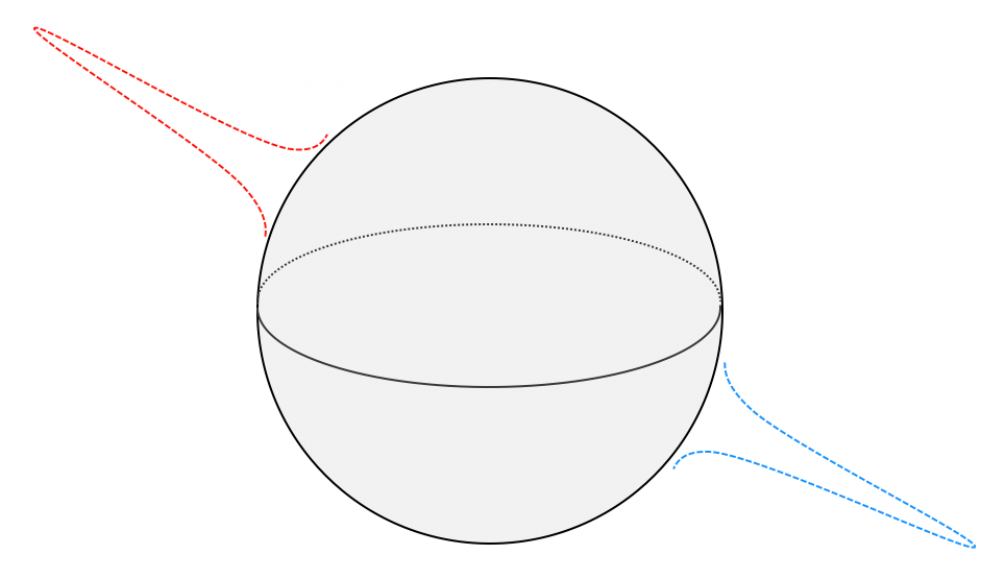}
	\caption{Schematic of source (red) and target (blue) distributions, in which the Optimal Transport PDE for any lens refractor problem is ill-posed with the cost function~\eqref{eq:lensp2ff}}
	\label{fig:dirac}
\end{figure}

Explicit examples for a continuous mapping $T$ which are unsolvable via a lens system can easily be furnished. Suppose we fix $x_0 \in \mathbb{S}^2$, then for any $x \in \mathbb{S}^2$, we define the source density:

\begin{equation}
f(x) = \frac{1}{\beta} e^{\frac{-\frac{1}{2}d^2(x, x_0)}{\sigma^2}},
\end{equation}
where $\beta$ is the normalization parameter:

\begin{equation}
\beta = 2\pi \int_{0}^{\pi} e^{-\frac{1}{2}(\theta/\sigma)^2}\sin \theta d \theta,
\end{equation}

and a target density given by:
\begin{equation}
g(x) = \frac{1}{\beta} e^{\frac{-\frac{1}{2}d^2(x, -x_0)}{\sigma^2}}.
\end{equation}
Choose $\sigma$ to be very small, for example $\sigma = 0.01$. Then, $\beta \approx 0.0006283$. In order to solve the lens refractor problem, we must satisfy the conservation of energy equation~\eqref{eq:conservation}. We choose a Borel set $B_{x_{0}}(5 \sigma)$ to be the open geodesic ball of radius $5 \sigma$ centered around $x_0$. Then:

\begin{equation}
\frac{2\pi}{\beta} \int_{0}^{0.05} e^{-\frac{1}{2}(\theta/\sigma)^2}\sin \theta d \theta \approx 0.999992.
\end{equation}

That is $99.9992\%$ of the source mass is contained within a radius $0.05$ of the point $x_0$. Likewise, for the target mass, we have $99.9992\%$ of the target mass is contained within a radius $0.05$ of the point $-x_{0}$. Thus, the condition that the mapping be mass preserving requires that the mapping $T$ satisfy:

\begin{equation}
0.999992 \approx \int_{B_{x_{0}}(0.05)} f(x)dx = \int_{T(B_{x_{0}}(0.05))} g(y)dy.
\end{equation}

By the fact that the mapping $T$ is continuous, the non-empty open set $B_{x_{0}}(0.05)$ is mapped to the non-empty open set $T(B_{x_{0}}(0.05))$. Since only $0.0008\%$ of the mass of $g$ is located outside the ball $B_{-x_{0}}(0.05)$, we must have that $T(B_{x_{0}}(0.05)) \cap B_{-x_{0}}(0.05) \neq \emptyset$ and the intersection is an open set. Thus, the continuous mapping $T$ satisfies $d_{\mathbb{S}^2}(x,T(x)) \geq \pi - 0.1$ in a non-negligible set and violates the fact that mass is not allowed to move too far. The result from~\cite{gutierrezhuang} shows, furthermore, that we cannot even expect weak solutions if the mass distributions require mass to move too far.

\subsection{Cost Functions Leading to Exponential-Type Maps}\label{sec:expo}
For the squared geodesic cost function $c(x,y) = \frac{1}{2}d_{\mathbb{S}^2}(x,y)^2$, a classical result of McCann~\cite{mccann} showed that the mapping is of the form $T(x, p) = \text{exp}_{x}(p)$. For other cost functions, like the one presented in this paper, $c(x,y) = -\log(n - x \cdot y)$, $n>1$, is it true that we can express the map as $T(x, p) = \text{exp}_{x}(\hat{p} \beta( \left\Vert p \right\Vert))$ for some function $\beta: \mathbb{R}^{+} \rightarrow \mathbb{R}$? This leads to the following definition:

\begin{definition}\label{def:expo}
We will say that a cost function leads to a map of exponential type if the Optimal Transport mapping can be expressed in the following way:

\begin{equation}
T(x) = \text{exp}_{x}\left(\frac{\nabla u(x)}{\left\Vert \nabla u(x) \right\Vert} \beta(\left\Vert \nabla u(x) \right\Vert) \right)
\end{equation}
where $u(x)$ is the potential function and $\beta$ is a continuous function on its appropriate domain.
\end{definition}

For lens refractor problems I and II, the cost functions lead to maps of exponential type. Since, Equation~\eqref{eq:lensmapping} can be expressed in the following way:

\begin{equation}\label{eq:r1r2}
T(x, p) = x R_1(\left\Vert p \right\Vert) + p R_2 (\left\Vert p \right\Vert),
\end{equation}
with $R_1(\left\Vert p \right\Vert)^2 + \left\Vert p \right\Vert^2 R_2(\left\Vert p \right\Vert)^2 = 1$, we see that $T(x, p)$ is on the great circle passing through $x$ in the direction $p$. By the formula in Equation~\eqref{eq:lensmapping}, the length of the geodesic is $\arccos R_1(\left\Vert p \right\Vert) = \arccos \left( \frac{\sqrt{1+(1-n^2)\left\Vert p \right\Vert^2} + n \gamma^2}{1 + \gamma^2} \right)= \beta(\left\Vert p \right\Vert)$. Thus, Equation~\eqref{eq:lensmapping} can be rewritten as:

\begin{equation}
T(x, p) = \text{exp}_{x} \left( \hat{p} \arccos \left( \frac{\sqrt{1+(1-n^2)\left\Vert p \right\Vert^2} + n \left\Vert p \right\Vert^2}{1 + \left\Vert p \right\Vert^2} \right) \right).
\end{equation}

For the values $n=1.333$, $n=1.52$, and $n=2.417$, typical values for the refractive indices of water, glass, and diamond, the $\beta$ functions are plotted in Figure~\ref{fig:beta}. Here we note that the ``worst" behavior we can expect from such functions is the appearance of a cusp, which is seen in the figure. Staying away from $p^{*}$, the $\beta$ functions are actually Lipschitz, see Lemma~\ref{thm:Lip} below.

\begin{figure}
\includegraphics[width=\textwidth]{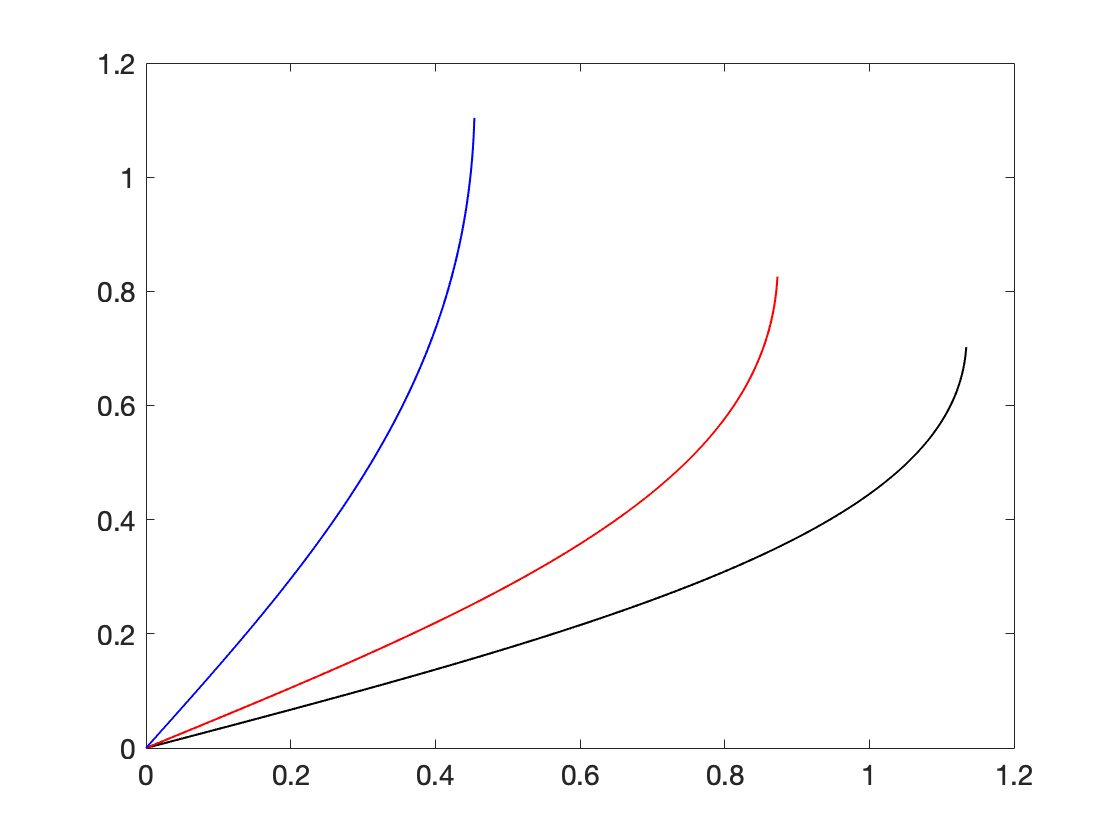}
\caption{The function $\beta$ plotted as a function of $\left\Vert p \right\Vert$. Black indicates $n=1.333$, approximately the refractive index of water, red indicates $n=1.52$, approximately the refractive index of glass, and blue indicates $n=2.417$, approximately the refractive index of diamond.}\label{fig:beta}
\end{figure}

For the cost function $c(x,y) = \frac{1}{\kappa} \log(\kappa x \cdot y - 1)$, we get the $\beta$ function:

\begin{equation}
\beta(\left\Vert p \right\Vert) = \arccos \left( \frac{\sqrt{\kappa^2 + (\kappa^2 - 1)\left\Vert p \right\Vert^2} + \left\Vert p \right\Vert^2}{\kappa(1 + \kappa^2 \left\Vert p \right\Vert^2)} \right).
\end{equation}

For the values $n=1.333$, $n=1.52$, and $n=2.417$, the $\beta$ functions are plotted in Figure~\ref{fig:beta2}. The $\beta$ functions are monotonically increasing, but flatten out as $\left\Vert p \right\Vert \rightarrow \infty$.

\begin{figure}
\includegraphics[width=\textwidth]{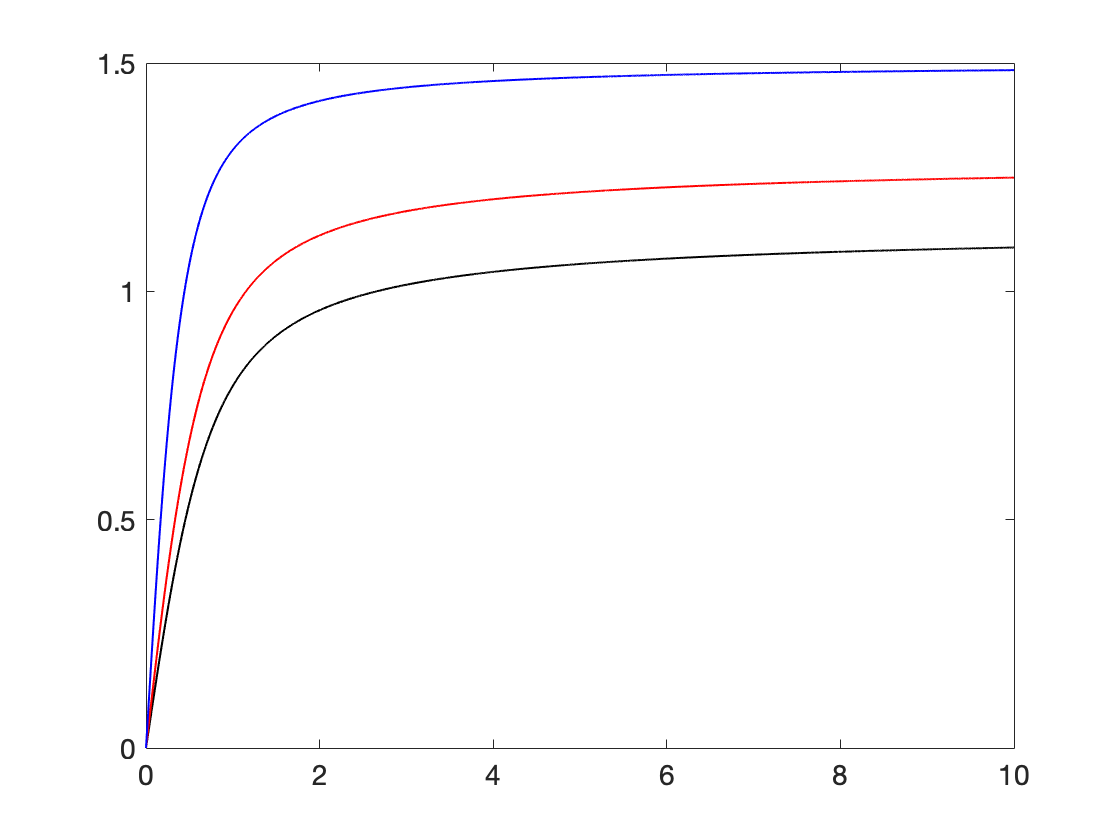}
\caption{The function $\beta$ plotted as a function of $\left\Vert p \right\Vert$. Black indicates $n=1.333$, red indicates $n=1.52$, and blue indicates $n=2.417$.}\label{fig:beta2}
\end{figure}

Thus, we say that $c(x,y) = -\log(n - x \cdot y)$ and $c(x,y) = \log(\kappa x \cdot y - 1)$ are cost functions leading to maps of exponential type. The following theorem shows when this is possible and when $\beta$ is monotone and satisfies $\beta(0) = 0$.

\begin{theorem}\label{thm:exponential}
Suppose the cost function on the unit sphere can be written as a function of the dot product between $x$ and $y$ and (consequently) the distance between $x$ and $y$, i.e. $c(x,y) = F(x \cdot y) = G(d_{\mathbb{S}^2}(x,y))$ and that $G$ is twice differentiable and $G''(z)$ is strictly positive or strictly negative on an open interval $(0,b)$ for $z \in (0, b)$ and $b\leq \pi$ and $F$ is differentiable and $F'(\zeta)\neq 0$ for $\zeta \in (\cos(b),1)$ and $\lim_{\zeta \rightarrow 1}F'(\zeta) \neq 0$. Then, the Optimal Transport mapping arising from this cost function can be written as

\begin{equation}
T(x, p) = \text{exp}_{x} \left( \hat{p} \beta(\left\Vert p \right\Vert) \right),
\end{equation}
where $\beta$ is monotonically increasing and $\beta(0) = 0$ and $\left\Vert p \right\Vert \in [0, \left\vert G'(b) \right\vert)$.
\end{theorem}

\begin{proof}
Suppose that the cost function can be written as $c(x,y) = F(x \cdot y)$. Equation~\eqref{eq:mapping}, i.e. $p = - \nabla_{x} F(x,y)$, for some tangent vector $p$ implies that we have $(\left\Vert p \right\Vert, 0) = \left( -F'(x \cdot y)y \cdot \hat{p}, -F'(x \cdot y) y \cdot \hat{q} \right)$ expressed in the tangent coordinates $(\hat{p}, \hat{q})$ for $\hat{q} \in T_{x}\mathbb{S}^2$ where $\hat{q} \cdot \hat{p} = 0$. This then immediately shows that $0 = y \cdot \hat{q} F'(x \cdot y)$. Thus, if $F'(x \cdot y) \neq 0$ for $x \cdot y \in (\cos b,1)$, then $y \cdot \hat{q} = 0$ (when $x \cdot y = 1$ we also have $y \cdot \hat{q} = 0$ by taking the limit) and therefore, the mapping does not go in the $\hat{q}$ direction and therefore, $T$ is of the form $T(x, p) = R_1 x + R_2 p$ and hence $T$ lies on the geodesic connecting $x$ to $T$.  Thus, the equation $\left\Vert p \right\Vert = -F'(x \cdot y) y \cdot \hat{p}$ tells us that as $\left\Vert p \right\Vert \rightarrow 0$, we have either $y \cdot \hat{p} \rightarrow 0$ or $F'(x \cdot y) \rightarrow 0$. If $y \cdot \hat{p} \rightarrow 0$, since $y$ is on the great circle including $x$ in the direction $p$, then either $x \cdot y \rightarrow 1$ or $x \cdot y =\rightarrow -1$. Likewise, if $F'(x \cdot y) \rightarrow 0$, then since $F'(x \cdot y) \neq 0$ for $x \cdot y \in (\cos b, 1)$, then $x \cdot y \rightarrow 1$. Therefore, as $\left\Vert p \right\Vert \rightarrow 0$, we find that $\beta(0) = 0$ and also $T$ can be written generally as:

\begin{equation}
T(x, p) = \text{exp}_{x} \left( \hat{p} \beta(\left\Vert p \right\Vert) \right),
\end{equation} 
where $\beta(\left\Vert p \right\Vert) = \arccos R_1(\left\Vert p \right\Vert)$ and exists for $\left\Vert p \right\Vert \in [0, \left\vert G'(b) \right\vert)$. Furthermore, if $F'(x \cdot y)<0$ on $(\cos b,1)$, then $y \cdot \hat{p} >0$, and likewise, if $F'(x \cdot y)>0$ on $(\cos b,1)$, then $y \cdot \hat{p} <0$.

Since $\cos \left( d_{\mathbb{S}^{2}}(x,y) \right) = x \cdot y$, we have that the cost function can be written as a function of the geodesic distance on the sphere $c(x,y) = F(x \cdot y) = G(d_{\mathbb{S}^2}(x,y))$. Letting $z = d_{\mathbb{S}^{2}}(x,y)$, we have that:

\begin{equation}
G'(z) = -F'(x \cdot y) \sin (z),
\end{equation}
since $\cos(z) = x \cdot y$. Furthermore, since $y = (x \cdot y)x + (y \cdot \hat{p})\hat{p}$, we get that $(y \cdot \hat{p})^2 = \sin^2(z)$. If $F'(x \cdot y)<0$, then $y \cdot \hat{p} = \sin (z)$ and if $F'(x \cdot y)>0$, then $y \cdot \hat{p} = -\sin(z)$.

Therefore, we find that

\begin{align}
\left\Vert p \right\Vert &= -F'(x \cdot y) y \cdot \hat{p} = G'(z), & \text{if} \ &F'(x \cdot y)<0\label{eq:Grelation1} \\
\left\Vert p \right\Vert &= -G'(z), & \text{if} \ &F'(x \cdot y)>0.\label{eq:Grelation2}
\end{align}

Thus, we get that $\beta$ and $\pm G'$ are inverses of each other, depending on the sign of $F'(x \cdot y)$. Taking a derivative, we get

\begin{align}
\beta'(\left\Vert p \right\Vert) &= 1/G''(z), & \text{if} \ &F'(x \cdot y)<0\label{eq:Gdouble1} \\
\beta'(\left\Vert p \right\Vert) &= -1/G''(z), & \text{if} \ &F'(x \cdot y)>0.\label{eq:Gdouble2}
\end{align}
Therefore, If $G''(z)$ is strictly positive or negative, we see that $\beta(\left\Vert p \right\Vert)$ is strictly monotone and differentiable.

\end{proof}

\begin{remark}
The preceding theorem can be modified to accommodate the case $\beta(0) = \pi$, as is the case with the logarithmic cost function $c(x,y) = -\log(1-x \cdot y)$, \textit{mutatis mutandis}. However, for the sake of exposition, in this paper we focus on the case where $\beta(0) = 0$, that is, when the potential function is flat with respect to the induced metric on the sphere, the Optimal Transport mapping satisfies $T(x)=x$.
\end{remark}


The proof of Theorem~\ref{thm:exponential} yielded three important relations. The first comes from the definition of $\beta$ and the exponential map and by denoting $z = d_{\mathbb{S}^2}(x,y)$:

\begin{equation}
z = \beta(\left\Vert p \right\Vert).
\end{equation}

The second relation is contained in Equations~\eqref{eq:Grelation1} and~\eqref{eq:Grelation2} and the third is contained in Equations~\eqref{eq:Gdouble1} and~\eqref{eq:Gdouble2}. Thus, we see that $\beta$ and $G$ are related very closely.

We saw that for the cost functions for the far-field lens refractor problem led to maps of exponential type. It should also be remarked here that the squared geodesic cost function $c(x,y) = d_{\mathbb{S}^2}(x,y)^2 = \arccos (x \cdot y)^2$ and the cost function arising from the reflector antenna: $c(x,y) = -\log \left\Vert x - y \right\Vert = -\log (1 - x \cdot y)$ lead to maps of exponential type, where $\beta(\left\Vert p \right\Vert)$ is monotonically increasing and monotonically decreasing, respectively. For a real-world example of a cost function on the sphere which is does not lead to a map of exponential type, we refer to the point-to-point reflector problem, whose setup and derivation are contained in~\cite{YadavThesis}. The cost function is:

\begin{equation}\label{eq:point2point}
c(x,y) = \log \left( \frac{1}{(T^2-L^2)^2} - \frac{1-x \cdot y}{2(T^2-L^2)(T-L(x \cdot \hat{e}))(T-L (y \cdot \hat{e}))} \right), \ \ x, y \in \mathbb{S}^2,
\end{equation}
where $T, L$ are parameters arising in the optical setup satisfying $T\geq L$. $T$ is the length of the optical path and $L$ is the distance between the points. There is a special direction in the problem which is the direction from the source point to the target point, denoted by $\hat{e}$. Due to the presence of the term $x \cdot \hat{e}$ arising from this special direction $\hat{e}$, the cost function cannot be written as simply a function of $x \cdot y$. For this reason, in computing the mapping for this kind of ``exotic" cost function, we observe that $T$ has the more general form:

\begin{equation}
T(x, p) = R_1 x + R_2 p + R_3 q,
\end{equation}
where w.l.o.g. $q = x \times p$. The presence of the nonzero $R_3$ term due to the cost~\eqref{eq:point2point} indicates that $T(x, p)$ does not travel along the great circle defined by $x$ and $p$. Thus, we have an example of a cost function which leads to a mapping $T$ which is not of exponential type.

\begin{remark}\label{rmk:Euclidean}
As we just saw, for the sphere, cost functions that are functions of the Riemannian distance between $x$ and $y$ can be written as functions of the dot product $x \cdot y$ and this simplifies the exposition and computations greatly. This is the only manifold where this is true. The natural extension, then, of costs leading to maps of exponential type in Euclidean space $c(x,y): \Omega \times \Omega ' \rightarrow \mathbb{R}$, where $\Omega, \Omega ' \subset \mathbb{R}^n$ are of the form $c(x,y) = H(d(x,y))$, where $d(x,y) = \left\Vert x - y \right\Vert$. We may write such cost functions as $c(x,y) = H(d(x,y)) = J(\frac{1}{2}d(x,y)^2)$ (where now the function $J$ will play the role of the earlier function $F$) and thus, $p = -\nabla_{x} c(x,y)$ results in the following equation:

\begin{equation}\label{eq:Euclidean}
p = J' \left( \frac{1}{2}d(x,y)^2 \right)(y - x),
\end{equation}
which indicates that $y - x = \frac{p}{J'}$, which we see implies that the mapping is along a geodesic from $x$ in the direction of $p$. The generalization, then, of the condition in Theorem~\ref{thm:exponential} that $F'(\zeta) \neq 0$ is given in the Euclidean case by $J'(\zeta) \neq 0$. Furthermore, we can find an expression for $\beta ' \left( \left\Vert p \right\Vert \right)$. By taking the magnitude of Equation~\eqref{eq:Euclidean} and using the definition of $J$ and $H$, we get:

\begin{equation}\label{eq:formula1}
\left\Vert p \right\Vert = J' \left( \frac{1}{2}d(x,y)^2 \right) d(x,y) = H'(d(x,y)),
\end{equation}
and thus, by letting $z = d(x,y)$ and assuming that $H''(z) \neq 0$, we can invert this and also show that:

\begin{equation}\label{eq:relationgeneralized}
\beta'(\left\Vert p \right\Vert) = \frac{1}{H''(z)},
\end{equation}
which is a result that held in the case of the sphere as well, see Equation~\eqref{eq:Gdouble1}. This shows that the condition $H''(z) \neq 0$ is also important in the Euclidean case. Furthermore, we have $z = \beta(\left\Vert p \right\Vert)$ and $\left\Vert p \right\Vert = H'(z)$ so $\beta$ and $H'$ are inverses.
\end{remark}

\begin{remark}\label{rmk:Riemannian}
The further generalization to Riemannian manifolds can be given as follows. If $c(x,y) = H(d_{M}(x,y)) = J(\frac{1}{2}d_{M}(x,y)^2)$, then since
\begin{equation}
p = -\nabla_{x} c(x,y) = -J' \left( \frac{1}{2}d_{M}(x,y)^2 \right) \nabla_{x} \left( \frac{1}{2}d_{M}(x,y)^2 \right),
\end{equation}
we get:

\begin{equation}
\frac{p}{J' \left( \frac{1}{2}d_{M}(x,y)^2 \right)} = -\nabla_{x} \left( \frac{1}{2}d_{M}(x,y)^2 \right),
\end{equation}
and thus, by the classical results of McCann~\cite{mccann}, we get
\begin{equation}
y = \text{exp}_{x}\left( \frac{p}{J'\left( \frac{1}{2}d_{M}(x,y)^2 \right)} \right).
\end{equation}
Thus, we see generally that if the cost function can be written as a function of the Riemannian distance on a manifold $M$, we get maps of exponential type. The generalization, then, of the condition in Theorem~\ref{thm:exponential} that $F'(\zeta) \neq 0$ is given in the Riemannian case by $J'(\zeta) \neq 0$. Since the distance from $x$ to $y$ is given by the magnitude of the tangent vector in the argument of the exponential map, we can, as in the Euclidean case, derive the relation:

\begin{equation}
\left\Vert p \right\Vert = J'\left( \frac{1}{2} d_{M}(x,y)^2 \right) d_{M}(x,y) = H'(d_{M}(x,y)),
\end{equation}
and thus, assuming $H''(z) \neq 0$, we derive Equation~\eqref{eq:relationgeneralized} for the Riemannian manifold case. As in the Euclidean case, we have $z = \beta(\left\Vert p \right\Vert)$ and $\left\Vert p \right\Vert = H'(z)$ so $\beta$ and $H'$ are inverses of each other.

\end{remark}

\subsection{Defective Cost Functions}\label{sec:defective}

For some cost functions of exponential type, the mass is unable to transport beyond a certain distance. We observed this with the cost functions $c(x,y) = -\log(n-x \cdot y)$ and $c(x,y) = \log(\kappa x \cdot y - 1)$ from the far-field refractor problem. Now, we identify properties of the cost function which lead to this phenomenon and call such cost functions defective.


\begin{definition}\label{def:defective}
A cost function $c(x,y): \mathbb{S}^2 \times \mathbb{S}^2 \rightarrow 0$, which can be written $c(x,y) = F(x \cdot y) = G(d_{\mathbb{S}^2}(x,y))$, is defective if $G''(z)=0$ or $G'(z) = \infty$ for some $z \in (0,\pi)$ and $G''(z)$ is either strictly positive or negative for $z \in [0, z^{*})$, where $z^{*} \in (0, \pi)$ is the smallest value for which $G''(z) = 0$ or $G'(z) = \infty$ and $F'(\zeta) \neq 0$ on the interval $(\cos z^{*},1)$ with $\lim_{\zeta \rightarrow 1}F'(\zeta) \neq 0$. If necessary in the discussion, we will refer to cost functions for which $G''(z^{*}) = 0$ as defective cost functions of the first type, and cost functions for which $G'(z^{*}) = \infty$ as defective cost functions of the second type. We also denote $p^{*}$ as the value $p^{*} = \left\vert G'(z^{*}) \right\vert$.
\end{definition}

Defective cost functions lead to maps of exponential type provided that the distance transported is less than $z^{*}$. By the definition, defective cost functions of the first type are those for which the concavity of the cost function as a function of the Riemannian distance on $\mathbb{S}^2$ changes for some $z \in (0, \pi)$. This change of concavity of the function was a noted issue going back to work by Gangbo and McCann~\cite{geometryOT}.

Defective cost functions of the second type are where $G'(z)$ diverges to infinity as $\left\Vert p \right\Vert \rightarrow \infty$. This implies by differentiating, that $G''(z^{*}) = \infty$. Then, the equations $G'(z) = \left\Vert p \right\Vert$ and $G''(z) = 1/\beta'(\left\Vert p \right\Vert)$ imply that since $G''(z) \neq 0$, we can invert $\left\Vert p \right\Vert = G'(z)$ to get $z = \beta(\left\Vert p \right\Vert)$ and thus $\lim_{\left\Vert p \right\Vert \rightarrow \infty} \beta(\left\Vert p \right\Vert) = z^{*}$ and $\lim_{\left\Vert p \right\Vert \rightarrow \infty} \beta'(\left\Vert p \right\Vert) = 0$. So, we see that defective cost functions $c(x,y) = G(z)$ restricted to the domain $z \in [0, z^{*})$ are ``well-behaved", but do not allow for mass to be transported beyond a distance $z^{*}$.

The cost functions $c(x,y) = d_{\mathbb{S}^2}(x,y)^2$ and $c(x,y) = -\log(1 - x \cdot y)$ are not defective costs, whereas $c(x,y) = -\log(n - x \cdot y)$ is a defective cost function of the first type because $G''(z^{*}) = 0$, for $z^{*} = \arccos(1/n)$. The cost function $c(x,y) = \frac{1}{\kappa}\log(\kappa x \cdot y - 1)$ is a defective cost function of the second type, because $G''(\arccos(1/\kappa)) = \infty$. For the cost $c(x,y) = -\log(n - x \cdot y) = -\log(n-\cos(z))$, we get that $G''(z) = 0$ when $z = \arccos(1/n)$. From Equation~\eqref{eq:lensrefractorconstraint}, we see that this value of $z$ is exactly the farthest the mapping may transport while staying real-valued. For the cost function $c(x,y) = -\log(n - x \cdot y)+\log(n+1)$, we have plotted the cost function as a function of geodesic distance and denoted the inflection points $G''(z) = 0$ with circles, see Figure~\ref{fig:costs}.

\begin{figure}
\includegraphics[width=\textwidth]{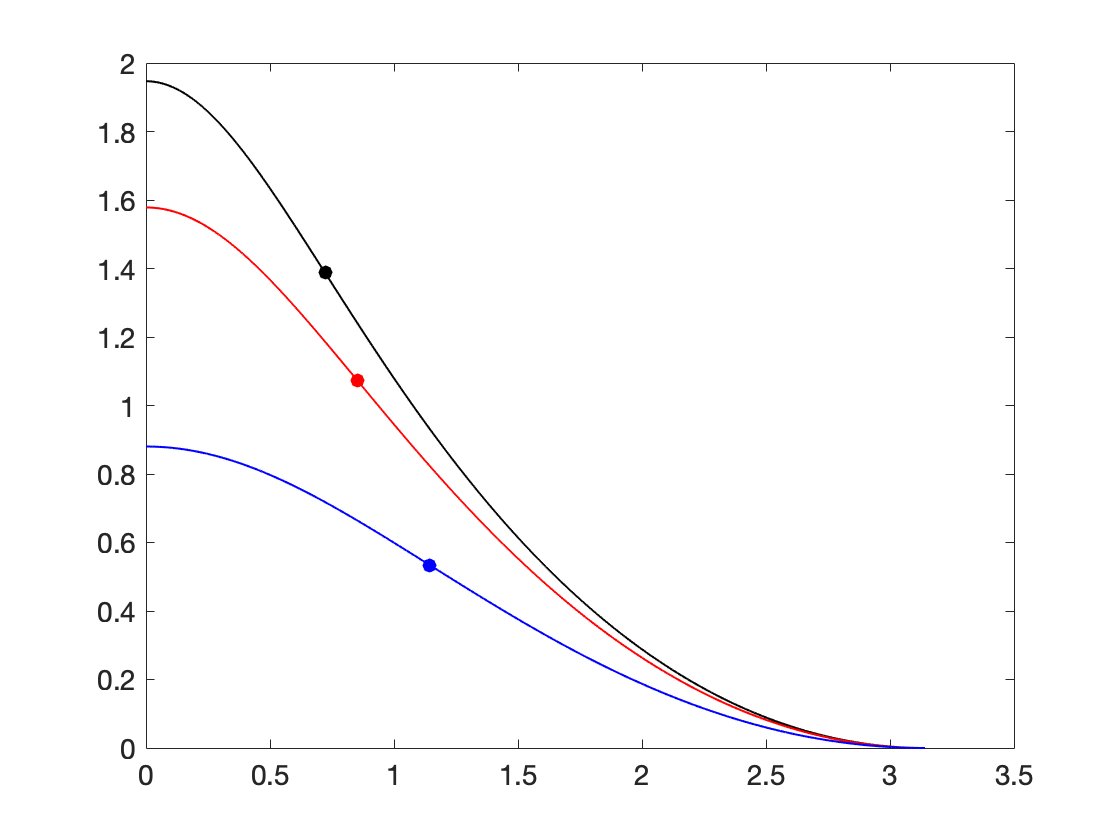}
\caption{Cost function $c(x,y) = -\log(n - \cos(z))+\log(n+1)$ plotted as a function of $z$. Black indicates $n=1.333$, approximately the refractive index of water, red indicates $n=1.52$, approximately the refractive index of glass, and blue indicates $n=2.417$, approximately the refractive index of diamond. Below the inflection points, the cost functions are concave.}\label{fig:costs}
\end{figure}


A defective cost function automatically fails to satisfy the injectivity for each $x \in \mathbb{S}^2$ of the map $y \rightarrow -\nabla_{x} c(x,y)$ over the set $\mathbb{S}^2 \setminus \left\{ -x \right\}$, which is one of the MTW conditions presented in~\cite{Loeper_OTonSphere}. It should be noted also that the cost functions in Theorem 4.1 of~\cite{Loeper_OTonSphere} lead to maps of exponential type that are not defective.

For $\left\Vert p \right\Vert \leq c < p^{*}$, we also have the important fact that the $\beta$ functions, defined in Definition~\ref{def:expo} are Lipschitz for defective cost functions below the value $p^{*}$.


\begin{lemma}\label{thm:Lip}
Suppose that there exists a $z^{*}$ such that $G''(z^{*}) = 0$. Denote by $p^{*}$ the value such that $\beta(p^{*}) = z^{*}$. The $\beta$ function arising from a defective cost function $c(x,y) = F(x \cdot y) = G\left( d_{\mathbb{S}^2}(x,y) \right)$, where $F'(\zeta) \neq 0$, is Lipschitz on the interval $[0,c]$, where $c<p^{*}$.
\end{lemma}

\begin{proof}
This is an immediate consequence of Equation~\eqref{eq:Gdouble1} or Equation~\eqref{eq:Gdouble2}. For $z < z^{*}$, we see that the derivative of $\beta$ is bounded.
\end{proof}

The important relation in Equations~\eqref{eq:Gdouble1} and~\eqref{eq:Gdouble2} explains why the $\beta$ function for the lens refractor problem I has a cusp and why the $\beta$ function for the lens refractor problem II flattens out. Therefore, if we have a $\beta$ function with a cusp, then it arises from a defective cost function of the first type, and if the $\beta$ function limits to a horizontal asymptote, where the asymptote is at a value strictly less than $\pi$, then it arises from a defective cost function of the second type. Both are types of cost functions which do not allow the mass to be transported a distance beyond a fixed value strictly less than $\pi$.

\begin{remark}\label{rmk:defectiveEuclidean}
It should be clear how to extend the definition of a defective cost function in Euclidean space. If we have a cost function $c(x,y) = H(\left\Vert x - y \right\Vert) = J \left( \frac{1}{2}\left\Vert x - y \right\Vert^2 \right)$, then denote $z^{*}$ to be the smallest value for which $H''(z) = 0$ or $H'(z) = \infty$. Then, let $F$ satisfy $F'(\zeta) \neq 0 $ for $\zeta < \sqrt{2z^{*}}$ and $H''(z)$ is either positive or negative for $z \in [0, z^{*})$. We call such a $c$ a defective cost function.
\end{remark}

\begin{remark}
In the case of more general Riemannian manifolds, if we have a cost function $c(x,y) = H(d_{M}(x,y)) = J \left( \frac{1}{2}d_{M}(x,y)^2 \right)$, then denote $z^{*}$ to be the smallest value for which $H''(z) = 0$ or $H'(z) = \infty$. Then, let $F$ satisfy $F'(\zeta) \neq 0 $ for $\zeta < \sqrt{2z^{*}}$ and $H''(z)$ is either positive or negative for $z \in [0, z^{*})$. We call such a $c$ a defective cost function. Note that in this case, $z^{*}$ will necessarily be less than the injectivity radius.
\end{remark}

\subsubsection{Using Euclidean Distance for Computations on Manifolds}

It is perhaps of interest to the computational Optimal Transport community to investigate the possibility of using the ambient Euclidean distance in a cost function, i.e. $c(x,y) = \frac{1}{2} d_{\mathbb{R}^3}(x,y)^2$ to solve the Optimal Transport problem on a connected, orientable surface $M \subset \mathbb{R}^3$. It is tempting to use the Euclidean distance is very cheap to compute, whereas the cost function $c(x,y) = \frac{1}{2}d_{M}(x,y)^2$, by contrast, is very expensive, since pairwise geodesic distances on $M$ must be computed for every point. Here, we demonstrate that the cost function $c(x,y) = \frac{1}{2} d_{\mathbb{R}^3}(x,y)^2 = H(d_{M}(x,y))$, while monotonic (as a function of the Riemannian distance) for many manifolds, is defective for $M=\mathbb{S}^2$. We have the cost function:

\begin{equation}
c(x,y) = \frac{1}{2} d_{\mathbb{R}^3}(x,y)^2 = 2 \sin^2 \left( \frac{d_{M}(x,y)}{2} \right),
\end{equation}
so $G(z) = 2 \sin^2 \left( \frac{z}{2} \right)$ and thus $G''(z) = \cos (z) = 0$ for $z = \pi/2$. Therefore, the cost function $c(x,y) = \frac{1}{2} d_{\mathbb{R}^3}(x,y)^2$ is defective. We see what happens when we solve for the mapping. Since for this cost function we have:

\begin{equation}
p = -\nabla_{x} c(x,y) = y - x,
\end{equation}
Taking the dot product of both sides with respect to $p$ and $q$, respectively, yields:

\begin{align}
\left\Vert p \right\Vert &= y \cdot p, \\
0 &= y \cdot q.
\end{align}

Using the fact that $y = (y \cdot p)\hat{p} + (y \cdot x)x$, we get:

\begin{equation}
y = p + \sqrt{1- \left\Vert p \right\Vert^2}x.
\end{equation}

We thus we that $\left\Vert p \right\Vert \leq 1$. The case $\left\Vert p \right\Vert = 1$ corresponds to $y = p$, which is the case where $d_{\mathbb{S}^2}(x,y) = \frac{\pi}{2}$. Thus, we see that for this cost function, mass is not allowed to move more than a distance of $\pi/2$.



We can also demonstrate the problem with some other surfaces, where the computation of the Riemannian distance on the manifold is not trivial computationally, for example the case of the $2$-torus $\mathbb{T}^2 \subset \mathbb{R}^3$ equipped with the standard Riemannian metric. While the cost function $c(x,y) = \frac{1}{2} d_{\mathbb{R}^3}(x,y)^2$ cannot be strictly written as a simply a function of the Riemannian distance on the manifold $M$, it shares the same property of defective cost functions in that it cannot transport mass beyond a certain distance. Although we have the $2$-torus in mind, it should be very clear how this argument extends to other surfaces and manifolds embedded in higher dimensions. The mapping satisfies:

\begin{equation}
p = -\nabla_{x} c(x,y) = y - x.
\end{equation}

Therefore, taking the dot product with respect to $p$ and $q$, respectively, we get:

\begin{align}
\left\Vert p \right\Vert &= y \cdot p - x \cdot p, \\
0 &= y \cdot q - x \cdot q.
\end{align}

Therefore, we have:

\begin{equation}
y = \left(\left\Vert p \right\Vert + x \cdot p \right) \hat{p} + \left( x \cdot q \right) \hat{q} + (y \cdot \hat{n})\hat{n}
\end{equation}

Thus, the resulting mapping arising from Optimal Transport on the torus with the squared Euclidean cost satisfies:

\begin{equation}
y = x + p + (y \cdot \hat{n} - x \cdot \hat{n})\hat{n}.
\end{equation}

In general, directly solving for $y \cdot \hat{n}$ yields a very complicated formula even when a level-set formulation of the surface is known. For the $2$-torus, the computation of $y \cdot \hat{n}$ in general requires one to solve a quartic equation in $y \cdot \hat{n}$. However, a simple geometric argument renders the need to compute $y \cdot \hat{n}$ unnecessary. The problem of solving for the mapping $y$ is equivalent to the problem of finding a particular point of intersection of the parametrized line $y = x + p + \alpha \hat{n}$ and the torus, since $y \in M$. In general, we can easily find $p$ which cannot lead to an Optimal Transport mapping as follows. Take the tangent plane $\mathcal{T}_x$ and orthogonally project the set $\mathbb{T}^2 \subset \mathbb{R}^3$ onto $\mathcal{T}_x$. That is, define $S: \mathbb{R}^3 \rightarrow \mathcal{T}_x$ by $S(z) = (1-\hat{n} \hat{n}^T) z $. Then, denote $T' = S(\mathbb{T}^2)$. For a fixed $x \in M$, a value of $p$ that is not allowed is one for which $x+p \notin T'$.

Even more generally, if $c(x,y) = f(d^2_{\mathbb{R}^3}(x,y))$, then the problem of finding $y$, then we have:

\begin{equation}\label{eq:general2}
p = -\nabla_{x} c(x,y) = f'(d^2_{\mathbb{R}^3}(x,y)) (y-x).
\end{equation}

Note that, by taking the magnitude of both sides, we get:

\begin{equation}
\frac{\left\Vert p\right\Vert}{f'(d^2_{\mathbb{R}^3}(x,y))} = d_{\mathbb{R}^3}(x,y).
\end{equation}

Therefore, taking the dot product of Equation~\eqref{eq:general2} with respect to $p$ and $q$, respectively, we get:

\begin{align}
\frac{\left\Vert p \right\Vert}{f'(d^2_{\mathbb{R}^3}(x,y))} &= y \cdot p - x \cdot p, \\
0 &= y \cdot q - x \cdot q.
\end{align}

Therefore, we have:

\begin{align}
y &= \left(\frac{\left\Vert p \right\Vert}{f'(d^2_{\mathbb{R}^3}(x,y))}+ x \cdot p \right) \hat{p} + \left( x \cdot q \right) \hat{q} + (y \cdot \hat{n})\hat{n} \\
y &= \left(\beta(\left\Vert p \right\Vert)+ x \cdot p \right) \hat{p} + \left( x \cdot q \right) \hat{q} + (y \cdot \hat{n})\hat{n}
\end{align}

Thus, the resulting mapping arising from Optimal Transport on the torus with the squared Euclidean cost satisfies:

\begin{equation}
y = \frac{\beta(\left\Vert p \right\Vert)}{\left\Vert p \right\Vert} p + x + (y \cdot \hat{n} - x \cdot \hat{n})\hat{n},
\end{equation}
where $z = \beta(\left\Vert p \right\Vert)$ which is found by inverting the equation $\left\Vert p \right\Vert = zf'(z^2)$, when it is invertible.
 
Thus, finding the mapping $y$ is equivalent to the problem of finding the closest point (to $x$) of intersection of the line $y = \frac{\beta(\left\Vert p \right\Vert)}{\left\Vert p \right\Vert} p + x + \alpha \hat{n}$ with the torus. Thus, we can see that no matter how $\beta$ increases as a function of $\left\Vert p \right\Vert$, it is not possible for the mapping to reach around to the ``other side" of the torus. Thus, we see that all such cost functions built from Euclidean distances are defective. This, of course, does not apply to the case of the sphere (on which the logarithmic cost function $c(x,y) = -2\log d_{\mathbb{R}^3}(x,y)$, for example, is not defective), since in that case, $x \cdot p = x \cdot q = 0$.

In the case where, \textit{a priori}, the mapping is known to be close to identity, a Euclidean cost can be used as an expedient for rapid computations, for example using the Sinkhorn iterations proposed by Cuturi~\cite{Cuturi}. It should be emphasized that this is not simply a claim that asymptotically the Euclidean distance approximates the Riemannian distance. Rather, the Euclidean distance can be used as a substitute distance up to a certain point.


\subsection{Derivation of the Mixed Hessian Term $\left\vert D^2_{xy} c(x,y) \vert_{y = T(x)} \right\vert$}

The mixed Hessian term $\left\vert D^2_{xy} c(x,y) \vert_{y = T(x)} \right\vert$ appearing in Equation~\eqref{eq:thePDE} is important to compute to check the MTW conditions as well as for numerical discretizations. See, for example, the paper~\cite{HT_OTonSphere2} for an example of a numerical discretization that uses the explicit form of the mixed Hessian. In this subsection, we derive various expressions for computing the mixed Hessian for cost functions leading to maps of exponential type. The expressions we derive, which are Equations~\eqref{eq:mH},~\eqref{eq:mHalt} and~\eqref{eq:mH2}, are for mappings of exponential type, which arise from cost functions depending solely on terms involving the dot product $x \cdot y$, see Section~\ref{sec:expo}, however, the derivation can be easily generalized to other, more general, cost functions also arising in optics applications and will be explored in future work.

Here we derive a formula for the mixed Hessian via an integral formulation. Using the notation established in Equation~\eqref{eq:r1r2}, define $R(p) = pR_2(\left\Vert p \right\Vert)$. Then, for a region $E \subset T_{x}\mathbb{S}^2$ define $U = R(E) = \left\{ pR_2(\left\Vert p \right\Vert) \vert p \in E \right\}$ and unit vectors $\hat{u}_1, \hat{u}_2$, where $\hat{u}_1 = \hat{p}$ and $\hat{u}_2 = x \times \hat{p}$. For any vector $v \in T_{x}\mathbb{S}^2$, this defines a coordinate system on $T_{x}\mathbb{S}^2$, i.e. $v = (u_1, u_2)$, where $u_1 = v \cdot \hat{u}_1$ and $u_2 = v \cdot \hat{u}_2$. Then, define the region $T(x,E) = \left\{ \sqrt{1 - u_1^2 - u_2^2} \vert (u_1, u_2) \in U \right\} \subset \mathbb{S}^2$, see Figure~\ref{fig:formula1}.

\begin{figure}[htp]
	\centering
	\includegraphics[width=\textwidth]{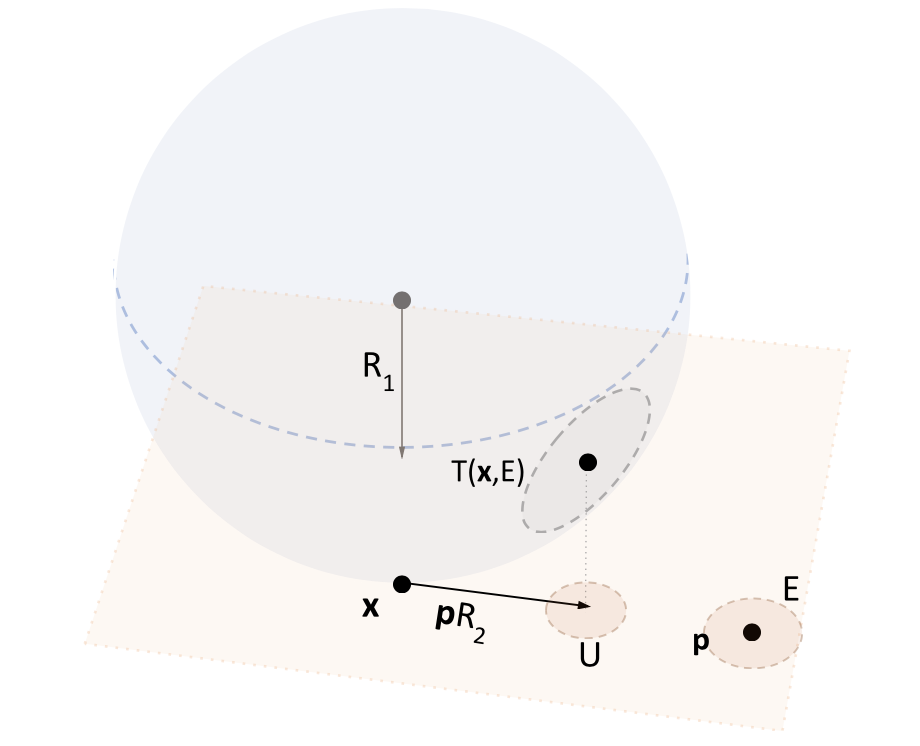}
	\caption{Change in area formula from tangent coordinates $p$ to coordinates on the sphere $T(x,p) = xR_1(\left\Vert p \right\Vert) + p R_2(\left\Vert p \right\Vert)$ via the coordinates $(u_1, u_2)$ of the orthogonal projection of $T(x,p)$ onto the tangent plane $T_{x}\mathbb{S}^2$.}
	\label{fig:formula1}
\end{figure}

Since the mixed Hessian satisfies the formula

\begin{equation}\label{eq:rel1}
\left\vert \det D^{2}_{xy} c \right\vert = \frac{1}{\left\vert \det D_{p} T \right\vert},
\end{equation}
we see that if we compute the quantity $\left\vert \det D_{p} T \right\vert$, then we can compute the determinant of the mixed Hessian. We see that the quantity $\left\vert \det D_{p} T \right\vert$ can be computed via computing the area of the region $T(x,E)$. From Figure~\ref{fig:formula1}, we see:

\begin{equation}\label{eq:rel2}
\int_{T(x,E)} dS = \int_{E} \left\vert \det D_{p}T \right\vert dp.
\end{equation}
Also,

\begin{equation}
\int_{T(x,E)} dS = \int_{U} \frac{1}{\sqrt{1- \left\Vert u \right\Vert^2}} du,
\end{equation}

\begin{equation}\label{eq:rel3}
= \int_{E} \frac{1}{\sqrt{1 - \left\vert R(p) \right\vert^2}} \left\vert \det \nabla R(p) \right\vert dp.
\end{equation}

And therefore, by Equations~\eqref{eq:rel1},~\eqref{eq:rel2} and~\eqref{eq:rel3}, we get the following expression for the mixed Hessian:

\begin{equation}
\left\vert \det D^2_{xy}c \right\vert = \frac{\sqrt{1 - \left\vert R(p) \right\vert^2}}{\left \vert \det \nabla R(p) \right\vert}.
\end{equation}

Since $R(p) = pR_2(\left\Vert p \right\Vert)$, we compute:

\begin{align}
(\nabla R)_{11} &= \frac{p_1^2}{\left\Vert p \right\Vert} R_2'(\left\Vert p \right\Vert) + R_2(\left\Vert p \right\Vert), \\
(\nabla R)_{12} &= (\nabla R)_{21} = \frac{p_1 p_2}{\left\Vert p \right\Vert} R_2'(\left\Vert p \right\Vert), \\
(\nabla R)_{22} &= \frac{p_2^2}{\left\Vert p \right\Vert} R_2'(\left\Vert p \right\Vert) + R_2(\left\Vert p \right\Vert),
\end{align}
and thus,

\begin{equation}
\det \left( \nabla R \right) = \left\Vert p \right\Vert R_2(\left\Vert p \right\Vert) R_2'(\left\Vert p \right\Vert) + R_2(\left\Vert p \right\Vert)^2,
\end{equation}
and hence

\begin{equation}\label{eq:mH}
\left\vert \det D^{2}_{xy} c \right\vert = \frac{\sqrt{1 - \left\Vert p \right\Vert^2 R_2(\left\Vert p \right\Vert)^2}}{\left\vert \left\Vert p \right\Vert R_2(\left\Vert p \right\Vert) R_2'(\left\Vert p \right\Vert) + R_2(\left\Vert p \right\Vert)^2 \right\vert}.
\end{equation}

We find that there is a more convenient expression for the determinant of the mixed Hessian in terms of $R_1(\left\Vert p \right\Vert)$. First, we use the fact that:

\begin{equation}
R_1(\left\Vert p \right\Vert)^2 + \left\Vert p \right\Vert^2 R_2(\left\Vert p \right\Vert)^2 = 1,
\end{equation}
and thus, we get $\left\vert R_1(\left\Vert p \right\Vert) \right\vert = \sqrt{1 - \left\Vert p \right\Vert^2 R_2(\left\Vert p \right\Vert)^2}$ and, also,

\begin{equation}
R_1(\left\Vert p \right\Vert)R_1'(\left\Vert p \right\Vert) + \left\Vert p \right\Vert R_2(\left\Vert p \right\Vert)^2 + \left\Vert p \right\Vert^2 R_2(\left\Vert p \right\Vert) R_2'(\left\Vert p \right\Vert) = 0,
\end{equation}
and thus,


\begin{equation}
\left\vert \left\Vert p \right\Vert R_2(\left\Vert p \right\Vert) R_2'(\left\Vert p \right\Vert)+R_2(\left\Vert p \right\Vert)^2 \right\vert = \frac{\left\vert  R_1(\left\Vert p \right\Vert) R_1'(\left\Vert p \right\Vert)\right\vert}{\left\Vert p \right\Vert}.
\end{equation}

Thus, in terms of $R_1$ we get the very simple expression:

\begin{equation}\label{eq:mHalt}
\left\vert \det D^{2}_{xy} c \right\vert = \frac{\left\Vert p \right\Vert}{\left\vert R_1'(\left\Vert p \right\Vert) \right\vert}.
\end{equation}

Using Equation~\eqref{eq:mH}, we derive a formula for the mixed Hessian term in terms of the function $\beta$ and $F'(x \cdot y)$ and $G''(d_{\mathbb{S}^2}(x,y))$. 





From the definition $T(x,p) = xR_1(\left\Vert p \right\Vert) + p R_2(\left\Vert p \right\Vert)$ and the fact that $\beta(\left\Vert p \right\Vert) = \arccos(R_1)$, we get that $\beta(\left\Vert p \right\Vert) = \arcsin(\left\Vert p \right\Vert R_2(\left\Vert p \right\Vert))$. Therefore, since

\begin{equation}
\beta'(\left\Vert p \right\Vert) = \frac{R_2(\left\Vert p \right\Vert) + xR_2'(\left\Vert p \right\Vert)}{\sqrt{1-\left\Vert p \right\Vert^2R_2(\left\Vert p \right\Vert)^2}},
\end{equation}
we get,

\begin{equation}
\left\vert \beta'(\left\Vert p \right\Vert) \right\vert = \frac{1}{\left\vert R_2(\left\Vert p \right\Vert) \right\vert}\left\vert D_p T \right\vert,
\end{equation}
from~\eqref{eq:mH}. Therefore,

\begin{equation}
\left\vert D_p T \right\vert = \left\vert R_2(\left\Vert p \right\Vert) \right\vert \left\vert \beta'(\left\Vert p \right\Vert) \right\vert = \frac{\sin \beta(\left\Vert p \right\Vert)}{\left\Vert p \right\Vert} \left\vert \beta'(\left\Vert p \right\Vert) \right\vert,
\end{equation}
and thus, we have the following expression for the mixed Hessian in terms of the $\beta$ function:

\begin{equation}\label{eq:mH2}
\left\vert D^2_{xy} c(x,y) \right\vert = \frac{\left\Vert p \right\Vert}{\sin \left( \beta(\left\Vert p \right\Vert) \right) \left\vert \beta'(\left\Vert p \right\Vert) \right\vert} 
\end{equation}

This expression, while cumbersome to compute by hand, proffers a much simpler route for checking the MTW conditions for all defective cost functions. Provided that we satisfy $z \in [0, z^{*})$, the mixed Hessian is strictly positive.

\begin{lemma}
For a defective cost function $c(x,y)$, when $z$ satisfies $z < z^{*}$, we have

\begin{equation}
\left\vert D^2_{xy} c(x,y) \right\vert > 0.
\end{equation}
\end{lemma}

\begin{proof}
Let $z = d_{\mathbb{S}^2}(x,y)$. Then, we derive the important equality:
\begin{equation}\label{eq:mixedHessianresult}
\left\vert D^2_{xy} c(x,y) \right\vert = \frac{\left\Vert p \right\Vert}{\sin \left( \beta(\left\Vert p \right\Vert) \right) \left\vert \beta'(\left\Vert p \right\Vert) \right\vert} = \frac{\left\vert G''(z) \right\vert}{\left\vert R_2(\left\Vert p \right\Vert) \right\vert} = \left\vert G''(z) \right\vert \left\vert F'(x \cdot y)\right\vert,
\end{equation}
since $y \cdot \hat{p} = \frac{\left\Vert p \right\Vert}{F'(x \cdot y)} = \left\Vert p \right\Vert R_2 \left( \left\Vert p \right\Vert \right)$, therefore, $R_2(\left\Vert p \right\Vert) = \frac{1}{F'(x \cdot y)}$. Since $\frac{1}{\left\vert \beta'(\left\Vert p \right\Vert) \right\vert} = \left\vert G''(z) \right\vert \neq 0$ for $z \in [0, z^{*})$, and $F'(\zeta) \neq 0$ for $\zeta \in (\cos z^{*}, 1)$ and $\lim_{\zeta \rightarrow 1}F'(\zeta) \neq 0$, we get that $\left\vert D^2_{xy} c(x,y) \right\vert > 0$.
\end{proof}


\begin{remark}\label{rmk:mHEuclidean}
A similar computation can be done for the Euclidean case for defective cost functions, see Remark~\ref{rmk:Euclidean} for the background and assumptions, i.e. $J'(\zeta) \neq 0$ and $H''(z) \neq 0$. It can be shown that for the $\mathbb{R}^d$ for such cost functions, we get:

\begin{equation}
\left\vert D_{p} T \right\vert = \beta ' (\left\Vert p \right\Vert) \left( \frac{\beta(\left\Vert p \right\Vert)}{\left\Vert p \right\Vert} \right)^{d}.
\end{equation}
Let $\zeta = \frac{1}{2}\left\Vert x - y \right\Vert^2$. By Equation~\eqref{eq:formula1} we get $J'(\zeta) z = \left\Vert p \right\Vert$ implies $J'(\zeta) = \left\Vert p \right\Vert / \beta(\left\Vert p \right\Vert)$. Using Equation~\eqref{eq:relationgeneralized}, we get:

\begin{equation}
\left\vert D^2_{xy} c(x,y) \right\vert =\frac{1}{\beta ' (\left\Vert p \right\Vert)} \left( \frac{\left\Vert p \right\Vert}{\beta(\left\Vert p \right\Vert)} \right)^{d} = \left\vert H''(z) \right\vert \left\vert J'(\zeta) \right\vert^{d},
\end{equation}
which is the natural equivalent of Equation~\eqref{eq:mixedHessianresult} and shows that the mixed Hessian term is nonzero for cost functions in Euclidean space such that $J'(\zeta) \neq 0$ and $H''(z) \neq 0$.
\end{remark}

\subsection{Cost-Sectional Curvature for Defective Cost Functions}\label{sec:costsectionalcurvature}

In this subsection, we present simple general formulas for checking the positivity of the cost-sectional curvature for cost functions of the type $c(x,y)= F(x \cdot y)$ on the sphere and $c(x,y) = J\left(\frac{1}{2} \left\Vert x - y \right\Vert^2 \right)$ on subsets of Euclidean space. These formulas, interestingly, will only depend just the first- and second-order derivatives of $F$ and $J$, respectively, (and an expression of $x\cdot y$ in terms of $\left\Vert p \right\Vert$) even though the cost-sectional curvature tensor requires taking four derivatives.

We will then use the formulas we derive to check the cost-sectional curvature for various cost functions, including the cost from the lens reflector problem I: $c(x,y) = -\log(n - x \cdot y)$ and the cost from the lens refractor problem II: $c(x,y) = \log(\kappa x \cdot y - 1)$. We will confirm that the cost function for the lens reflector problem I does not satisfy the positive cost-sectional curvature condition, but does for the lens reflector problem II, as was found in~\cite{gutierrezhuang}. However, we emphasize that the formulas we derive will be valid for any defective cost function and as such extend beyond the lens refractor problem.

The cost-sectional curvature condition comes from defining the following fourth-order tensor, which was first defined in~\cite{MTW}, but we will use the formula from~\cite{LoeperReg}. Note: there is a minus sign in front of the cost function and that the derivatives with respect to $x$ need to be taken with respect to the metric on either $\mathbb{S}^2$ or $\mathbb{R}^d$, respectively. Here, we define the cost-sectional curvature tensor:

\begin{definition}\label{def:costsectionalcurvature}
On the domain that the map $y \mapsto -\nabla_{x} c(x,y)$ is injective, $\left\vert D^2_{xy} c(x,y) \right\vert \neq 0$ and $c(x,y) \in C^4$, we define the map

\begin{equation}
\mathfrak{G}_{c} (\xi, \eta) = D^4_{p_{k}, p_{l}, x_{i}, x_{j}} \left[ (x, p) \rightarrow -c(x, y)(p) \vert_{y = T(x)} \right] \xi_{i} \xi_{j} \eta_{k} \eta_{l}.
\end{equation}
\end{definition}

An important condition to check for the purposes of regularity theory is positive sectional curvature, usually known as the $\text{Aw}$ condition, i.e. that on an appropriate domain for all $x \in M$, $y \in M$, $\xi \in T_{x} M$, $\eta \in T_{x} M$, $\xi \perp \eta$ that:

\begin{equation}
\mathfrak{G}_{c}(\xi, \eta) \geq 0.
\end{equation}

The $\text{Aw}$ condition is an important condition to check such that smooth source and target mass density functions which are bounded away from zero and infinity lead to Optimal Transport mapping and potential functions that are smooth as well. Without the \text{Aw} condition, it is possible to have smooth smooth source and target mass density functions which are bounded away from zero and infinity and possibly have an Optimal Transport mapping that is not even continuous, see~\cite{LoeperReg} and also~\cite{figalli}.

A more stringent condition is known as the $\text{As}$ condition, which is often the condition used in deriving more precise regularity results, as in ~\cite{Loeper_OTonSphere}. It requires that there exist a constant $C_0>0$ such that

\begin{equation}
\mathfrak{G}_{c}(\xi, \eta) \geq C_0 \left\vert \xi \right\vert^2 \left\vert \eta \right\vert.
\end{equation}


\subsubsection{Computation of the cost-sectional curvature for $c(x,y) = F(x \cdot y)$ on the sphere $\mathbb{S}^2$}\label{sec:cscsphere}

To compute the Hessian on $\mathbb{S}^2$, we first compute the $3 \times 3$ matrix Hessian for the cost function $c(x,y) = F(x \cdot y)$ using Euclidean derivatives, since our metric on the sphere is induced by the surrounding Euclidean space. Given a function $f: \mathbb{R}^3 \rightarrow \mathbb{R}$, we compute the Hessian as follows:

\begin{equation}
\nabla^{2}_{xx} f(x) = D^2_{xx} f(x) - (D_{x} f(x) \cdot x) \text{Id},
\end{equation}
where $D$ are the standard Euclidean (ambient) derivatives. We will thus compute the term $\nabla^{2}_{xx} c(x,y) \vert_{y = T}$. We get:

\begin{equation}
D_{x} F(x \cdot y) \cdot x = F'(x \cdot y) x \cdot y = -\frac{\left\Vert p \right\Vert}{\sin \beta (\left\Vert p \right\Vert)} \cos \beta(\left\Vert p \right\Vert) = -\left\Vert p \right\Vert \cot \beta(\left\Vert p \right\Vert).
\end{equation}

We also compute:

\begin{equation}
\frac{\partial^2}{\partial x_{i} x_{j}} F(x \cdot y) = F''(x \cdot y) y_{i} y_{j}.
\end{equation}

Thus, we get:

\begin{equation}
\nabla^{2}_{xx} F(x \cdot y) = F''(x \cdot y) \begin{pmatrix} y_{1}^2 & y_1y_2 & y_1 y_3 \\ y_{1}y_2 & y_2^2 & y_2 y_3 \\ y_{1}y_3 & y_2y_3 & y_3^2 \end{pmatrix} - F'(x \cdot y) (x \cdot y) \begin{pmatrix} 1 & 0 & 0 \\ 0 & 1 & 0 \\ 0 & 0 & 1 \end{pmatrix}.
\end{equation}


Recall that $y_{i} = x_i (x \cdot y) + \frac{p_i}{F'(x \cdot y)}$. This Hessian can be entirely expressed as a function of $x, p$, however, for simplicity, we express in terms of $x, \zeta, p$, where $\zeta = x \cdot y = \cos \beta(\left\Vert p \right\Vert)$:

\begin{equation}\label{eq:HessSphere}
\left( \nabla^{2}_{xx} c(x, y) \vert_{y = T} \right)_{kl} = F''(\zeta) \left(x_{k} \zeta + \frac{p_{k}}{F'(\zeta)}  \right)\left(x_{l} \zeta + \frac{p_{l}}{F'(\zeta)}  \right) - \zeta F'(\zeta) \delta_{kl},
\end{equation}

Since $\xi$ and $p$ can be related via $\zeta = \cos \beta(\left\Vert p \right\Vert)$, we relabel Equation~\eqref{eq:HessSphere} by $f_1(p) = F''(\zeta) \zeta^2$, $f_2(p) = F''(\zeta)/(F'(\zeta))^2$, $f_3(p) = \zeta/ F'(\zeta)$, and $f_4(p) = -\zeta F'(\zeta)$ and get

\begin{equation}
\left( \nabla^{2}_{xx} c(x, y) \vert_{y = T} \right)_{kl} = f_1(p) x_{k}x_{l} + f_2(p) p_{k}p_{l} + f_3(p) (x_kp_l + x_l p_k) + f_4(p) \delta_{kl}.
\end{equation}

With this simplification of notation, we proceed and compute:

\begin{multline}
D_{p_{i}}\left( \nabla^{2}_{xx} c(x, y) \vert_{y = T} \right)_{kl} = D_{p_{i}}f_1(p) x_{k}x_{l} + D_{p_{i}}f_2(p) p_{k}p_{l} + f_2(p)(\delta_{ik} p_{l} + \\\delta_{il}p_{k}) + D_{p_{i}}f_3(p) (x_kp_l + x_l p_k) + f_3(p)(x_{k} \delta_{il} + x_{l} \delta_{ik}) + D_{p_{i}}f_4(p) \delta_{kl}.
\end{multline}

Taking another derivative, we get:

\begin{multline}
D_{p_{i}p_{j}}\left( \nabla^{2}_{xx} c(x, y) \vert_{y = T} \right)_{kl} = D_{p_{i}p_{j}}f_1(p) x_{k}x_{l} + D_{p_{i}p_{j}} f_2(p) p_k p_l + D_{p_{i}} f_2(p) (\delta_{jk}p_l + \delta_{jl}p_{k}) + \\
D_{p_{j}} f_2(p) (\delta_{ik}p_{l} + \delta_{il}p_{k}) + f_2(p) (\delta_{ik}\delta_{jl} + \delta_{il}\delta_{jk}) + D_{p_{i}p_{j}}f_3(p)(x_k p_l + x_l p_k) + \\ D_{p_{i}}f_3(p) (x_k \delta_{jl} + x_{l} \delta_{jk}) + 
D_{p_{j}} f_3(p)(x_k \delta_{il} + x_{l} \delta_{ik}) + D_{p_{i}p_{j}}f_4(p) \delta_{kl}.
\end{multline}

Now, we compute the following:

\begin{multline}
\sum_{i, j, k, l} D_{p_{i}p_{j}}\left( \nabla^{2}_{xx} c(x, y) \vert_{y = T} \right)_{kl} \xi_{i} \xi_{j} \eta_{k} \eta_{l} = \left\langle D^2_{pp}f_1(p) \xi, \xi \right\rangle (x \cdot \eta)^2 + \left\langle D^2_{pp}f_2(p) \xi, \xi \right\rangle (p \cdot \eta)^2 + \\
4(D_{p} f_2(p) \cdot \xi)(\xi \cdot \eta)(p \cdot \eta) + f_2(p) (\xi \cdot \eta)^2 + 2\left\langle D^2_{pp}f_3(p) \xi, \xi \right\rangle (x \cdot \eta)(p \cdot \eta) + 2(D_{p} f_3(p) \cdot \xi)(x \cdot \eta)(\xi \cdot \eta) + \\
2(D_{p}f_3(p) \cdot \xi)(\xi \cdot \eta)(x \cdot \eta) + 
\left\langle D^2_{pp}f_4(p) \xi, \xi \right\rangle \left\vert \eta \right\vert^2.
\end{multline}

We now use the fact that $\xi \perp \eta = 0$ and $x \cdot \eta = 0$:

\begin{multline}
\sum_{i, j, k, l} D_{p_{i}p_{j}}\left( \nabla^{2}_{xx} c(x, y) \vert_{y = T} \right)_{kl} \xi_{i} \xi_{j} \eta_{k} \eta_{l} =  \left\langle D^2_{pp}f_2(p) \xi, \xi \right\rangle (p \cdot \eta)^2 +  \left\langle D^2_{pp}f_4(p) \xi, \xi \right\rangle \left\vert \eta \right\vert^2.
\end{multline}

This shows that in order to satisfy condition $\text{As}$, we must check two conditions, (1) the function $f_4(p)$ must be strictly negative definite (strictly concave as a function of $\left\Vert p \right\Vert$) and (2) $f_2(p) + f_4(p)$ must be strictly negative definite (strictly concave as a function of $\left\Vert p \right\Vert$). In order to satisfy condition $\text{Aw}$, we simply weaken the strictly negative definite condition to be just simply negative definite.

We can use these formulas to check the cost-sectional curvature conditions $\text{Aw}$ and $\text{As}$ for various cost functions on the sphere:

\begin{enumerate}
\item For the far-field refractor problem I, we replace $f_2(p) = F''(x \cdot y)/(F'(x \cdot y))^2$ and $f_4(p) = -(x \cdot y) F'(x \cdot y)$ back in, we can now check the lens problem: 

\begin{equation}
f_4(p) = -(x \cdot y) F'(x \cdot y) = -\frac{x \cdot y}{n - x \cdot y},
\end{equation}
and

\begin{equation}
f_2(p) = F''(x \cdot y)/(F'(x \cdot y))^2 = -1.
\end{equation}

Recall from Equation~\eqref{eq:lensmapping} that

\begin{equation}
x \cdot y = \frac{\sqrt{1 + (1-n^2) \left\Vert p \right\Vert^2}+n \left\Vert p \right\Vert^2}{1 + \left\Vert p \right\Vert^2}.
\end{equation}

The result, for different $n$ is shown in Figure~\ref{fig:f4}. All are convex. This means that the lens problem does not satisfy condition $\text{Aw}$, corroborating the result in~\cite{gutierrezhuang}.

\begin{figure}
\includegraphics[width=\textwidth]{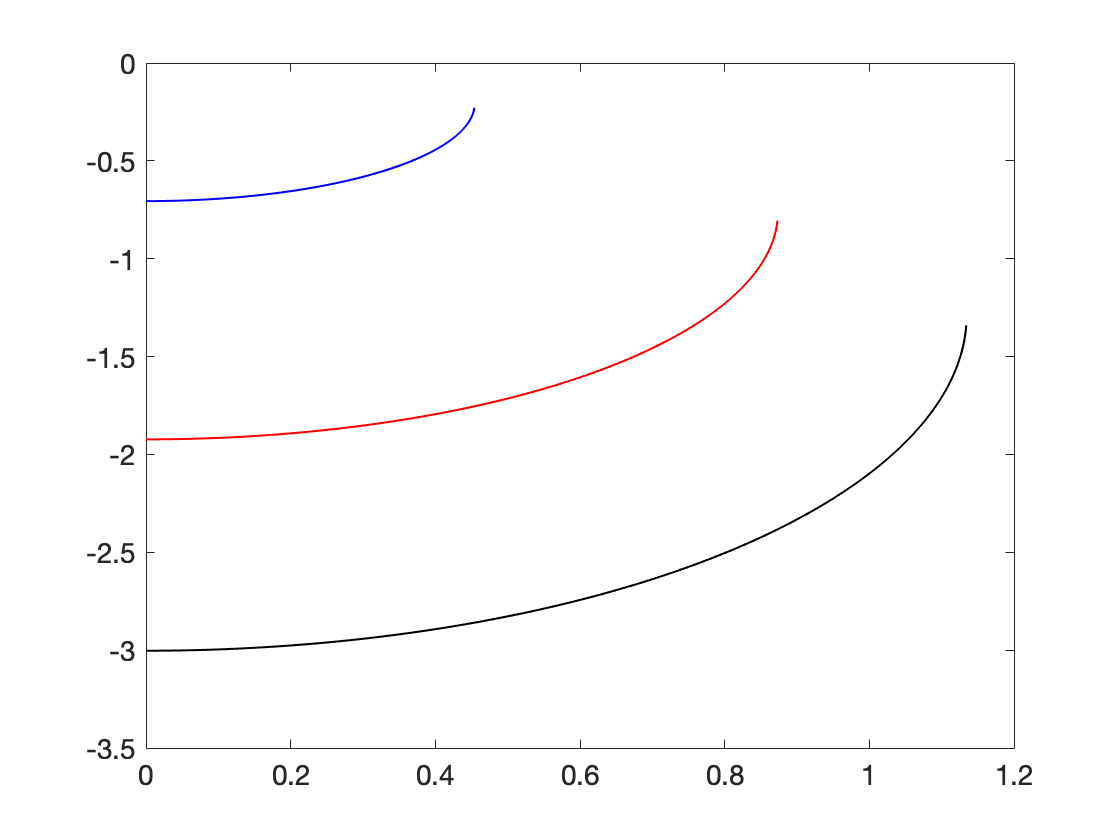}
\caption{The function $f_4(p)$ plotted as a function of $\left\Vert p \right\Vert$. Black indicates $n=1.333$, approximately the refractive index of water, red indicates $n=1.52$, approximately the refractive index of glass, and blue indicates $n=2.417$, approximately the refractive index of diamond.}\label{fig:f4}
\end{figure}

\item For the far-field refractor problem II, we compute:

\begin{equation}\label{eq:kappaf4}
f_4(p) = -(x \cdot y) F'(x \cdot y) = -\frac{\kappa x \cdot y}{\kappa (x \cdot y) - 1},
\end{equation}
and

\begin{equation}
f_2(p) = F''(x \cdot y)/(F'(x \cdot y))^2 = -1.
\end{equation}

Then, we use Equation~\eqref{eq:kappamapping} to get:

\begin{equation}\label{eq:kappaxy}
x \cdot y = \frac{\sqrt{\kappa^2+(\kappa^2-1)\left\Vert p \right\Vert^2} + \left\Vert p \right\Vert^2}{\kappa(1 + \left\Vert p \right\Vert^2)}.
\end{equation}

The result, in Figure~\ref{fig:kappaf4} shows that $f_4(p)$ is concave in $\left\Vert p \right\Vert$ and thus satisfies condition $\text{Aw}$. In order to check $\text{As}$, we need to compute the second derivative of the function in Equation~\eqref{eq:kappaf4} with respect to $\left\Vert p \right\Vert$ using Equation~\eqref{eq:kappaxy}. This is a straightforward exercise, and a similar computation was done in~\cite{gutierrezhuang}, where it was shown that condition $\text{As}$ does not hold as $\left\Vert p \right\Vert \rightarrow \infty$. However, if we restrict $\left\Vert p \right\Vert < p^{*}$, then $f_{4}(p)$ is strictly convex and thus $\text{As}$ can be satisfied.

\begin{figure}
\includegraphics[width=\textwidth]{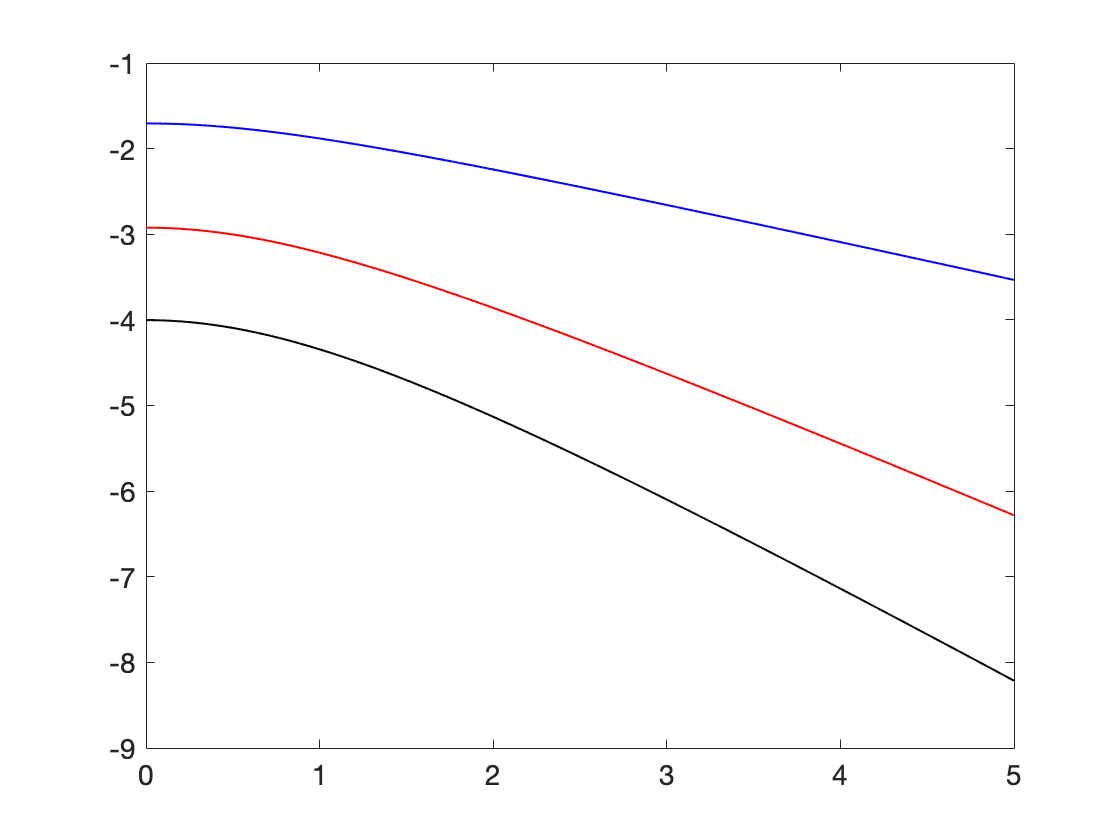}
\caption{The function $f_4(p)$ plotted as a function of $\left\Vert p \right\Vert$. Black indicates $\kappa=1.333$, red indicates $\kappa=1.52$, and blue indicates $\kappa=2.417$.}\label{fig:kappaf4}
\end{figure}

\item For the squared geodesic cost $c(x,y) = \frac{1}{2} d_{\mathbb{S}^2}(x,y)^2$, we have that $F(x \cdot y) = \frac{1}{2} \arccos(x \cdot y)^2$. Therefore, we get:

\begin{equation}
f_4(p) = -(x \cdot y) F'(x \cdot y) = \frac{(x \cdot y) \arccos(x \cdot y)}{\sqrt{1-(x \cdot y)^2}} = \frac{\left\Vert p \right\Vert \cos \left\Vert p \right\Vert}{\sqrt{1 - \cos^2 \left\Vert p \right\Vert}},
\end{equation}
which is concave in $\left\Vert p \right\Vert$. Also,

\begin{equation}
f_2(p) = F''(x \cdot y)/(F'(x \cdot y))^2 = \frac{1}{(\arccos(x \cdot y))^2} - \frac{x \cdot y}{\arccos(x \cdot y)\sqrt{1-(x \cdot y)^2}},
\end{equation}
which is convex in $\left\Vert p \right\Vert$, but the sum $f_2(p) + f_4(p)$ is concave in $\left\Vert p \right\Vert$ as expected. Checking the condition $\text{As}$ is a straightforward exercise (take two derivatives), and the results from~\cite{Loeper_OTonSphere} show that the condition $\text{As}$ is satisfied for the squared geodesic cost function.
\end{enumerate}

\subsubsection{Computation of the cost-sectional curvature for $c(x,y) = F(x \cdot y)$ for Euclidean space $\mathbb{R}^d$}\label{sec:cscEuclidean} We can easily check the cost-sectional curvature conditions $\text{Aw}$ and $\text{As}$ in Euclidean space. This can be contrasted with the previous computation on the sphere in Section~\ref{sec:cscsphere}. Denote $\zeta =  \frac{1}{2}\left\Vert x - y \right\Vert^2$. We get:

\begin{equation}
\frac{\partial}{\partial x_{i}} c(x,y) = J' \left( \zeta \right) (x_i - y_i),
\end{equation}
and thus
\begin{equation}
\frac{\partial^2}{\partial x_{i} \partial x_{j}} c(x,y) = J''(\zeta) (x_i - y_i)(x_j - y_j),
\end{equation}
and
\begin{equation}
\frac{\partial^2}{\partial x_{i}^2} c(x,y) = J''(\zeta) (x_i - y_i)^2 + J'(\zeta),
\end{equation}
and thus,

\begin{multline}\label{eq:HessEuclidean}
\left( \nabla^{2}_{xx} c(x, y) \vert_{y = T} \right)_{kl} = J'( \zeta) \delta_{kl} + \\
J''(\zeta) \begin{pmatrix} (x_1-y_1)^2 & (x_1 - y_1)(x_2 - y_2) & (x_1 - y_1)(x_3 - y_3) \\ (x_1 - y_1)(x_2 - y_2) & (x_2 - y_2)^2 & (x_2 - y_2)(x_3 - y_3) \\ (x_1 - y_1)(x_3 - y_3) & (x_2 - y_2)(x_3 - y_3) & (x_3 - y_3)^2 \end{pmatrix}.
\end{multline}

Since $x - y = -\frac{p}{J'(\zeta)}$, we get:

\begin{equation}\label{eq:HessEuclidean}
\left( \nabla^{2}_{xx} c(x, y) \vert_{y = T} \right)_{kl} = J'( \zeta) \delta_{kl} - \frac{d}{d \zeta} \left( \frac{1}{J'(\zeta)} \right) p_{k}p_{l}.
\end{equation}

Since we can relate $p$ and $\zeta$ via the equation $\left\Vert p \right\Vert = J'(\zeta) \sqrt{2\zeta}$, we may rewrite Equation~\eqref{eq:HessEuclidean} as:

\begin{equation}
\left( \nabla^{2}_{xx} c(x, y) \vert_{y = T} \right)_{kl} = f_1(p) \delta_{kl} + f_2(p) p_{k}p_{l},
\end{equation}
where $f_1(p) = J'(\zeta)$ and $f_2(p) = \frac{J''(\zeta)}{J'(\zeta)^2}$. Thus, we can compute:

\begin{multline}
D_{p_i, p_j} \left( \nabla^{2}_{xx} c(x, y) \vert_{y = T} \right)_{kl} = D_{p_{i} p_{j}}f_1(p)\delta_{kl} + D_{p_{i}p_{j}}f_2(p)p_{k}p_{l} + \\
D_{p_{i}}f_{2}(p) (\delta_{jk}p_{l} + \delta_{jl}p_{k}) + D_{p_{j}}f_2(p)(\delta_{ik}p_{l}+ \delta_{il}p_{k}) + f_2(p)(\delta_{ik}\delta_{jl}+\delta_{il}\delta_{jk}).
\end{multline}

Thus, we compute

\begin{multline}
\sum_{i, j, k, l} D_{p_i, p_j} \left( \nabla^{2}_{xx} c(x, y) \vert_{y = T} \right)_{kl} \xi_{i} \xi_{j} \eta_{k} \eta_{l} = \left\langle D_{pp}f_1(p)\xi, \xi \right\rangle \left\vert \eta \right\vert^2 + \left\langle D_{pp} f_2(p) \xi, \xi \right\rangle (p \cdot \eta)^2 + \\
4(D_{p} f_2(p) \cdot \xi)(p \cdot \eta)(\xi \cdot \eta)  + 2f_2(p)(\xi \cdot \eta)^2.
\end{multline}

Since $\xi \perp \eta$, we get:

\begin{equation}
\sum_{i, j, k, l} D_{p_i, p_j} \left( \nabla^{2}_{xx} c(x, y) \vert_{y = T} \right)_{kl} \xi_{i} \xi_{j} \eta_{k} \eta_{l} = \left\langle D_{pp}f_1(p)\xi, \xi \right\rangle \left\vert \eta \right\vert^2 + \left\langle D_{pp} f_2(p) \xi, \xi \right\rangle (p \cdot \eta)^2.
\end{equation}

As in the case of the sphere in Section~\ref{sec:cscsphere}, in order for condition $\text{As}$ to hold, it is necessary for two properties to hold: (1) that $f_1$ is strictly concave in $\left\Vert p \right\Vert$ and that (2) $f_1 + f_2$ is strictly concave in $\left\Vert p \right\Vert$. For property $\text{Aw}$ to hold, we weaken strict concavity to simply concavity. We have

\begin{equation}
f_1(p) = J'(\zeta) = \frac{\left\Vert p \right\Vert}{\beta(\left\Vert p \right\Vert)},
\end{equation}
and
\begin{equation}
f_2(p) = \frac{J''(\zeta)}{J'(\zeta)^2} = \frac{1}{\left\Vert p \right\Vert} \left( \frac{1}{\left\Vert p \right\Vert \beta'(\left\Vert p \right\Vert)} - \frac{1}{\beta(\left\Vert p \right\Vert)} \right).
\end{equation}

As an example, for the squared Euclidean distance cost, $J(\zeta) = \zeta = \frac{1}{2}\left\Vert p \right\Vert^2$, for which $J'(\zeta) = 1$, which is concave in $\left\Vert p \right\Vert$. Also, $J''(\zeta) = 0$, so we see that we satisfy the cost-sectional curvature condition $\text{Aw}$, but $\text{As}$ does not hold, which is corroborated by~\cite{TW}.

\begin{remark}
Some explicit computations for more general Riemannian manifolds have been done in, for example,~\cite{figalli} where there is an interesting example where for the squared geodesic cost function $c(x,y) = \frac{1}{2}d_{M}(x,y)^2$, the cost-sectional curvature is not positive for an ellipsoid of revolution. For general cost functions and general Riemannian manifolds, the formula for the cost-sectional curvature tensor is formidable, see, for example the survey paper by Villani~\cite{villanisurvey}.
\end{remark}

\section{Regularity and Solvability}\label{sec:regularity}
\subsection{Ma-Trudinger-Wang Conditions for Defective Cost Functions}\label{sec:MTW}

It was proved, in~\cite{gutierrezhuang} that there exist weak solutions to the lens refractor problem I \& II, provided that, most importantly, the conditions in Equations~\eqref{eq:lensrefractorconstraint} and~\eqref{eq:restriction2} were met, respectively. In~\cite{Karakhanyan}, it was proved that smooth solutions exist for the lens refractor problem II, with an argument using supporting ellipsoids and hyperboloids, respectively. Here, instead, we are inspired to take the route taken in~\cite{Loeper_OTonSphere} and prove a regularity result for the lens refractor problem II, but, more generally, for any defective cost function that satisfies the MTW conditions.

We begin by stating the MTW conditions, formulated originally in~\cite{MTW}, but we focus on the Riemannian generalization as stated in~\cite{Loeper_OTonSphere}. Given a compact domain $D \subset \mathbb{S}^2 \times \mathbb{S}^2$, denote by $\pi_1: \mathbb{S}^2 \times \mathbb{S}^2 \mapsto \mathbb{S}^2$, the projection $\pi_1(x, y) = x$ and its inverse $\pi_{1}^{-1}(x) = \left\{ x \right\} \times \mathbb{S}^2$. For any $x\in \pi_{1}(D)$, we denote by $D_{x}$ the set $D \cap \pi_{1}^{-1}(x)$. Then, we introduce the following conditions:

\begin{hypothesis}
\begin{itemize}
\item [$\mathbf{A0}$] The cost function $c$ belongs to $C^{4}(D)$.
\item [$\mathbf{A1}$] For all $x \in \pi_{1}(D)$, the map $y \rightarrow -\nabla_{x}c(x,y)$ is injective on $D_{x}$.
\item [$\mathbf{A2}$] The cost function $c$ satisfies $\det D^{2}_{xy}c \neq 0$ for all $(x,y)$ in $D$.
\item [$\mathbf{Aw}$] The cost-sectional curvature is non-negative on $D$. That is, for all $(x,y) \in D$, for all $\xi, \eta \in T_{x}\mathbb{S}^2$, $\xi \perp \eta$, 
\begin{equation}
\mathfrak{G}_{c}(x, y)(\xi, \eta) \geq 0
\end{equation}
\item [$\mathbf{As}$] The cost-sectional curvature is uniformly positive on $D$. That is, for all $(x,y) \in D$, for all $\xi, \eta \in T_{x}\mathbb{S}^2$, $\xi \perp \eta$, 
\begin{equation}
\mathfrak{G}_{c}(x, y)(\xi, \eta) \geq C_0 \left\vert \xi \right\vert^2 \left\vert \eta \right\vert^2
\end{equation}
\end{itemize}
\end{hypothesis}

Denoting $z^{*} = \left\{ \min z: G''(z)=0 \ \text{or} \ G''(z) = \infty, z \in (0,\pi) \right\}$, we now choose the subdomain $D$ as follows. For each $x \in \mathbb{S}^2$, we denote the corresponding geodesic ball $B_{x}(z^{*})$. Then, let $0<\gamma<z^{*}$
\begin{equation}
D_{\gamma} = \cup_{x} \left\{ x \right\} \times B_{x}(\gamma)
\end{equation}

Now that we have defined $D_{\gamma}$, the computations of Section~\ref{sec:computations} allow for us to verify the MTW conditions. Assuming that $F \in C^4([\cos \gamma, 1])$, $G \in C^4([0, \gamma])$, then, $c$ satisfies $\mathbf{A0}$ on $D_{\gamma}$. As long as $c$ satisfies the hypotheses in Theorem~\ref{thm:exponential}, then $\mathbf{A1}$ is satisfied on $D_{\gamma}$. By the computation in Equation~\eqref{eq:mixedHessianresult}, if we have $F'(\zeta) \neq 0$ and $G''(z) \neq 0, \infty$, then $c$ satisfies $\mathbf{A2}$ on $D_{\gamma}$. As shown in Section~\ref{sec:costsectionalcurvature}, if $f_4(p)$ is concave and $f_4(p) + f_2(p)$ is concave for $\left\Vert p \right\Vert < p^{*}$, then $\mathbf{Aw}$ is satisfied for $c$ on $D_{\gamma}$. The condition $\mathbf{As}$ can then be checked by taking two derivatives of $f_2(p)$ and $f_4(p)$ with respect to $\left\Vert p \right\Vert$.

\begin{remark}
Based on Remarks~\ref{rmk:Euclidean},~\ref{rmk:defectiveEuclidean},~\ref{rmk:mHEuclidean}, and the work in Section~\ref{sec:cscEuclidean}, the MTW conditions can be checked for a cost function on Euclidean space in an analogous subdomain $D_{\gamma} \subset \mathbb{R}^d \times \mathbb{R}^d$.
\end{remark}

\subsubsection{Example of Computations for Cost Functions in Theorem 4.1 of Loeper~\cite{Loeper_OTonSphere}: Power Costs on the Unit Sphere $c(x,y) = \frac{1}{p}d_{\mathbb{S}^2}(x,y)^p$}

As remarked in Section~\ref{sec:introduction}, an important result of the work in this paper is to allow for the MTW conditions to be checked easily for a wide class of cost functions, not just defective cost functions. Note that if $z^{*} \geq \pi$, then our cost functions naturally fit into Theorem 4.1 of Loeper~\cite{Loeper_OTonSphere}, which are the cost functions that can be checked with the preexisting regularity theory of Loeper. The benefit is that in Section~\ref{sec:computations} we have identified simple conditions on the cost function which allow for the hypotheses of Theorem 4.1 of Loeper (including the MTW conditions) to be verified. As an illustration of this, we verify the MTW conditions for the power cost functions on the sphere: $c(x,y) = \frac{1}{s}d_{\mathbb{S}^2}(x,y)^s = \frac{1}{s} \arccos(x \cdot y)^s$.

\begin{enumerate}
\item First, we check that the power cost functions satisfy some preliminary conditions from Theorem 4.1 in Loeper~\cite{Loeper_OTonSphere}. That is, $G(z)$ is smooth and strictly increasing with $G'(0) = 0$. We immediately verify that $G(z) = \frac{1}{s} z^{s}$ is smooth and strictly increasing for $s>0$ and $G'(z) = z^{s-1}$, so that means that if $s>1$ we have that $G'(0) = 0$. From now on assume that $s>1$.

\item Now, we verify the MTW conditions on $D = \mathbb{S}^2 \times \mathbb{S}^2 \setminus \text{antidiag}$ using the formulas from this manuscript. Clearly, $\mathbf{A0}$ is satisfied on $D$.

\item In order for $\mathbf{A1}$ to be satisfied, we check the hypotheses of Theorem~\ref{thm:exponential}. We need $F'(\zeta) \neq 0$ for $\zeta \in (-1,1]$ and $G''(z)$ strictly positive or negative on $(0, \pi)$. Now, $F(\zeta) = \frac{1}{s} \arccos(\zeta)^s$. Therefore, $F'(\zeta) = -\arccos(\zeta)^{s-1}/\sqrt{1-\zeta^2}$. Thus, we have $F'(\zeta) \neq 0$ for $\zeta \leq 1$ if $s\leq 2$. Since $G''(z) = (s-1)z^{s-2}$, we have that $G''(z)$ is strictly positive on $(0, \pi)$ for $s>1$. Therefore, condition $\mathbf{\text{A1}}$ is satisfied for $1<s\leq2$.

\item The condition $\mathbf{A2}$, by Equation~\eqref{eq:mixedHessianresult} is satisfied for $1<s\leq2$.

\item In order to verify condition $\mathbf{As}$, we compute the term $R_1(\left\Vert p \right\Vert)$ from Equation~\eqref{eq:r1r2}. This is simple, using the fact that $\beta$ and $G'$ are inverses and therefore $\left\Vert p \right\Vert = G'(z) = z^{s-1}$. Therefore, $z = \left\Vert p \right\Vert^{1/(s-1)}$. Therefore, we get $R_1(\left\Vert p \right\Vert) = \cos \left( \left\Vert p \right\Vert^{\frac{1}{s-1}} \right)$. We must check that $f_4(p)$ and $f_2(p) + f_4(p)$ are strictly convex as a function of $\left\Vert p \right\Vert$, as shown in Section~\ref{sec:costsectionalcurvature}. We have

\begin{equation}
f_4(p) = -\zeta F'(\zeta) = \frac{\zeta \arccos(\zeta)^{s-1}}{\sqrt{1-\zeta^2}}
\end{equation}
where $\zeta = x \cdot y$. Thus,

\begin{equation}
f_4(p)=\frac{ \left\Vert p \right\Vert \cos\left( \left\Vert p \right\Vert^{\frac{1}{s-1}} \right)}{ \sqrt{1-\cos^2\left( \left\Vert p \right\Vert^{\frac{1}{s-1}} \right)}}, \ \ \ \left\Vert p \right\Vert \in [0, \pi^{s-1}]
\end{equation}
and
\begin{multline}
f_2(p) = F''(\zeta)/(F'(\zeta))^2 = \\
\left\Vert p \right\Vert^{\frac{s}{s-1}} \left( (s-1) - \left\Vert p \right\Vert^{\frac{1}{s-1}} \frac{\cos \left( \left\Vert p \right\Vert^{\frac{1}{s-1}} \right)}{\sqrt{1-\cos^2 \left( \left\Vert p \right\Vert^{\frac{1}{s-1}} \right)}} \right), \ \ \ \left\Vert p \right\Vert \in [0, \pi^{s-1}]
\end{multline}

\begin{equation}
= (s-1)\arccos(\zeta)^{-s} - \arccos(\zeta)^{2-2s}f_4(p)
\end{equation}

\begin{equation}
= \frac{s-1}{\left\Vert p \right\Vert^{\frac{s}{s-1}}} - \frac{1}{\left\Vert p \right\Vert^2}f_4(p)
\end{equation}
and therefore,

\begin{equation}
f_2(p) + f_4(p) = \frac{s-1}{\left\Vert p \right\Vert^{\frac{s}{s-1}}} + \frac{\left\Vert p \right\Vert^2-1}{\left\Vert p \right\Vert^2}f_4(p)
\end{equation}

This can be used to confirm, for example, that condition $\mathbf{As}$ does not hold for $s>2$, since for small values of $\left\Vert p \right\Vert$, we have:

\begin{equation}
f_2(p) + f_4(p) = (s-2)\left\Vert p \right\Vert^{\frac{-s}{s-1}} + o \left( \left\Vert p \right\Vert^{\frac{-s}{s-1}} \right)
\end{equation}
so the cancellation of the highest order term only happens when $s=2$. The fact that the power costs for $s \neq2$ do not satisfy the $\mathbf{As}$ condition agrees with the result for Euclidean power costs in~\cite{TW}.
\end{enumerate}

\subsection{Conditions on Source and Target Mass Density Functions that Constrain the Mapping}

Now that we have verified the MTW conditions on a subdomain $D_{\gamma}$, we need to make sure that our source and target densities $f$ and $g$ are such that they do not require mass to move too far, as was the case in the example in Section~\ref{sec:illposedness}. The work in~\cite{Loeper_OTonSphere} showed that certain technical conditions on the source and target mass distributions restricted mass to move away from the cut locus on the sphere. Therefore, in that paper, the MTW conditions were shown to be satisfied on the subdomain $D \subset \mathbb{S}^2 \times \mathbb{S}^2$, where the domain $D$ was chosen to be $D = \mathbb{S}^2 \times \mathbb{S}^2 \setminus \text{antidiag}$, where the set $\text{antidiag}$ is defined as $(x, -x) \in \mathbb{S}^2 \times \mathbb{S}^2$.

In this section, we show that more strict conditions can be found on the source and target mass distributions that restrict the Optimal Transport mapping to satisfy $d_{\mathbb{S}^2}(x,T) \leq \gamma$ for any desired $\gamma < \pi$, where the Optimal Transport mapping $T$ arises from certain defective cost functions, such as the cost function arising in the lens refractor problem. This then allows us to use the regularity framework in~\cite{Loeper_OTonSphere} to build a regularity theory for a large class of defective cost functions.

Define $G(\phi)(x) = \text{exp}_{x} (\nabla \phi(x))$, which is the Optimal Transport mapping arising from the squared geodesic cost function. Also define $T(\phi)(x) = \text{exp}_{x} \left( \frac{\nabla \phi(x)}{\left\Vert \nabla \phi(x) \right\Vert} \beta \left( \nabla \phi(x) \right) \right)$ where $\beta$ is the function in Theorem~\ref{thm:exponential} arising from a defective cost function. 
As before, let $z^{*}$ denote the smallest value for which $G''(z) = 0$ or $G'(z) = \infty$. First, we begin with a result due to Delano\"{e} and Loeper~\cite{LoeperDelanoe}.

\begin{theorem}[Delano\"{e} and Loeper]\label{thm:dl}
Let $\phi :\mathbb{S}^2 \rightarrow \mathbb{R}$ be a $C^3$ functions such that $G(\phi)$ is a diffeomorphism and let $\rho: \mathbb{S}^2 \rightarrow \mathbb{R}$ be the positive $C^1$ functions defined by:

\begin{equation}
G(\phi)_{\#} \text{dVol} = \rho \text{dVol},
\end{equation}
where $\text{dVol}$ is the standard $2$-form on the unit sphere $\mathbb{S}^2$ induced by the standard Euclidean metric on $\mathbb{R}^3$. Then,

\begin{equation}
\max_{\mathbb{S}^2} \left\vert \nabla \phi \right\vert \leq C \max_{\mathbb{S}^2} \left\vert \nabla [\rho^{-1}] \right\vert.
\end{equation}
\end{theorem}

We continue by defining a $c$-convex function.

\begin{definition}
A function $\phi: \mathbb{S}^2: \rightarrow \mathbb{R}$ is $c$-convex if at each point $x \in \mathbb{S}^2$ there exists a point $y \in \mathbb{S}^2$ and a value $\phi^{c}(y)$ such that:

\begin{align}
-\phi^{c}(y) - c(x,y) &= \phi(x) \\
-\phi^{c}(y) - c(x',y) &\leq \phi(x'), \ \ \ \forall x' \in \mathbb{S}^2.
\end{align}
\end{definition}

\begin{theorem}\label{thm:bounds}
Let $\phi :\mathbb{S}^2 \rightarrow \mathbb{R}$ be a $C^3$, $c$-convex function, for the defective cost function $c$, such that $G(\phi)$  and $T(\phi)$ are diffeomorphisms and the Jacobian $\left\vert \nabla T \circ G^{-1} \right\vert$ has a bounded derivative and let $\tilde{\rho}: \mathbb{S}^2 \rightarrow \mathbb{R}$ be the positive $C^1$ function defined by:

\begin{equation}
T(\phi)_{\#} \text{dVol} = \tilde{\rho} \text{dVol}.
\end{equation}

Then, there exists a constant $\tilde{C}>0$ such that:

\begin{equation}
\max_{\mathbb{S}^2} \left\vert \nabla \phi \right\vert \leq \tilde{C} \left( \max \tilde{\rho} + \max \left\vert \nabla \tilde{\rho} \right\vert \right) .
\end{equation}
\end{theorem}

\begin{proof}
Note: if $\rho = \tilde{\rho} = 1$, then $\phi$ is $c$-convex, $G(\phi)$ and $T(\phi)$ are diffeomorphisms and $\left\vert \nabla T \circ G^{-1} \right\vert$ has a bounded derivative.

From Theorem~\ref{thm:dl}, the $c$-convex function is $C^3$ and $G(\phi)$ is a diffeomorphism. Therefore, $\phi$ satisfies the bound:

\begin{equation}
\max_{\mathbb{S}^2} \left\vert \nabla \phi \right\vert \leq C \max_{\mathbb{S}^2} \left\vert =\nabla [\rho^{-1}] \right\vert.
\end{equation}
where $\rho$ is defined by:

\begin{equation}
G(\phi)_{\#} \text{dVol} = \rho \text{dVol},
\end{equation}


Since $T(\phi)$ is a diffeomorphism, it satisfies $0<\left\vert \nabla T(\phi)(x) \right\vert<\infty$. The function $\tilde{\rho}$ then satisfies $\tilde{\rho} = (T \circ G^{-1})_{\#} \rho$ and likewise $\rho = (G \circ T^{-1})_{\#} \tilde{\rho}$. We denote $S = T \circ G^{-1} (x) = \text{exp}_{x}\left( \frac{\nabla \phi(x)}{\left\Vert \nabla \phi(x) \right\Vert}\left( - \nabla \phi(x) + \beta \left( \nabla \phi(x) \right) \right) \right)$. 
We have then, that:

\begin{equation}
\tilde{\rho} = \rho(S^{-1}(x)) \left\vert \nabla S^{-1}(x) \right\vert.
\end{equation}

Since $S$ is invertible and a diffeomorphism, since it is the composition of diffeomorphisms, then we have $c \leq \left\vert \nabla S^{-1}(x) \right\vert \leq C$, and since $G$ is a diffeomorphism, we have $0<m \leq \rho \leq M$. Therefore, 

\begin{equation}
\tilde{\rho} \geq c m.
\end{equation}

Then, since

\begin{equation}
\rho = \tilde{\rho}(S(x)) \left\vert \nabla S(x) \right\vert,
\end{equation}
we get:

\begin{equation}
\rho \geq C \min \tilde{\rho} \geq cC m,
\end{equation}
and therefore,
\begin{equation}
\frac{1}{\rho} \leq \frac{1}{cC m}.
\end{equation}

Now, since $\left\vert \nabla S \right\vert$ has bounded derivatives, there exists a $\mu>0$ such that:

\begin{equation}
\left\vert \nabla \rho \right\vert = \left\vert \nabla S^{-1}_{\#} \tilde{\rho} \right\vert = \left\vert \nabla \left[ \tilde{\rho}(S(x)) \left\vert \nabla S(x) \right\vert \right] \right\vert
\end{equation}

\begin{equation}
= \left\vert \nabla \tilde{\rho}(S(x)) \left\vert \nabla S(x) \right\vert + \tilde{\rho}(S(x)) \nabla \left\vert \nabla S(x) \right\vert \right\vert \leq  \left\vert c \nabla \tilde{\rho}(S(x))+\mu \tilde{\rho}(S(x))  \right\vert,
\end{equation}
Therefore, we get:

\begin{equation}
\left\vert \nabla \phi \right\vert \leq \left\vert \nabla [\rho^{-1}] \right\vert \leq \tilde{C} \left( \max \tilde{\rho} + \max \left\vert \nabla \tilde{\rho} \right\vert \right).
\end{equation}
\end{proof}

This result shows that provided that we control the ratio of $g/f$ and the ratio of the derivative of $g/f$, then the derivative of the potential function $\phi$ can be controlled as well. This, in turn, via the fact that for a defective cost function, as long as $\left\Vert p \right\Vert < p^{*}$,  means that the Optimal Transport mapping will not transport a distance greater than $z^{*}$ from Definition~\ref{def:defective}. At this point, we can adapt the results of Loeper in ~\cite{Loeper_OTonSphere} to the case of defective cost functions to get the following result:

\begin{theorem}\label{thm:regularity}
Let $\mu$, $\nu$ be two $C^{\infty}$ probability measures that are strictly bounded away from zero and $u$ a $c$-convex potential function, where $c$ is a defective cost function, such that, defining $T(u)(x) = \text{exp}_{x} \left( \frac{\nabla u(x)}{\left\Vert \nabla u(x) \right\Vert} \beta(\left\Vert \nabla u(x) \right\Vert) \right)$, we have $T(u)_{\#}\mu = \nu$ and also the bound $\left\vert \nabla u \right\vert < p^{*}$, where $p^{*}$ is defined in Definition~\ref{def:defective}. Additionally, let $c$ satisfy assumption $\mathbf{As}$ on $D_{\gamma}$. Then $u \in C^{\infty}(\mathbb{S}^2)$.
\end{theorem}

\begin{proof}
The proof follows the line of reasoning in Section $5$ of Loeper~\cite{Loeper_OTonSphere}. First, we have shown by the estimate in Theorem~\ref{thm:bounds} that there exist potential functions whose derivatives can be bounded by $g/f$ and the derivative of $g/f$. Assuming that $\mathbf{As}$ is satisfied, we use a $C^2$ estimate established in~\cite{MTW} on $u$. Then, we use the method of continuity to construct smooth solutions for any $\mu$, $\nu$ for any smooth positive densities satisfying the hypotheses of Theorem~\ref{thm:bounds}.
\end{proof}

By applying the above theorem to the case of refractor problem II, and keeping in mind that the condition $\text{As}$ is only true for the cost function $c(x,y) = -\log(\kappa x \cdot y - 1)$ for all $\left\Vert p \right\Vert<p^{*}$, we arrive at the following result for the refractor problem II:

\begin{corollary}
Let $\mu$, $\nu$ be two $C^{\infty}$ probability measures that are strictly bounded away from zero. Then, the potential function $u$ for the refractor problem II is $C^{\infty}$.
\end{corollary}

\begin{remark}
One interesting consequence of Theorem~\ref{thm:regularity} is that it suggests that some lens refractor problems may be solved using a series of $N$ lenses $\left\{ L_{i} \right\}_{i=1, \dots, N}$, in the case that it is impossible to solve using a single lens. 

\end{remark}







\section{Conclusion}\label{sec:conclusion}
In examining the far-field lens refractor problem and its solvability condition, we found that we could answer many theoretical and computational questions by formulating a theory for defective cost functions on the unit sphere. Using these definitions, we examined when the MTW conditions held for such cost functions and derived formulas for verifying the cost-sectional curvature for defective cost functions. We discussed how to extend these ideas to Euclidean space and some ideas to general Riemannian manifolds. We examined a sufficient condition on the source and target densities $f$ and $g$, respectively, that ensured the mapping $T$ satisfied the solvability bound, which was achieved by bounding the ratio $g/f$ and the derivative of $g/f$ in the sup-norm. Using this, we could employ the regularity framework of Loeper in order to establish $C^{\infty}$ smoothness for the potential function for defective cost functions satisfying the strictly positive cost-sectional curvature condition.

\textbf{Acknowledgements}: I would like to especially thank Brittany Hamfeldt for introducing me to the lens refractor problem and working alongside for some of the initial exploratory computations. I would also like to thank Rene Cabrera for listening to and discussing the ideas as they developed.

\textbf{Data Availability Statement}: The author declares that the data supporting the findings of this study are available within the paper.

\bibliographystyle{plain}
\bibliography{ThreeSystems}


\end{document}